\documentclass[10pt]{article} % use larger type; default would be 10pt

\usepackage[utf8]{inputenc} % set input encoding (not needed with XeLaTeX)

\usepackage[bottom = 4cm]{geometry} % to change the page dimensions
\usepackage{graphicx} % support the \includegraphics command and options

\usepackage{booktabs} % for much better looking tables
\usepackage{array} % for better arrays (eg matrices) in maths
\usepackage{paralist} % very flexible & customisable lists (eg. enumerate/itemize, etc.)
\usepackage{verbatim} % adds environment for commenting out blocks of text & for better verbatim
\usepackage{subfig} % make it possible to include more than one captioned figure/table in a single float
\usepackage{amsmath,amssymb,amsthm}

\usepackage{mathtools}
\usepackage{microtype}
\usepackage{xcolor}

\usepackage[ruled,lined,norelsize]{algorithm2e}%norelsize only for amsart

\usepackage{fancyhdr} % This should be set AFTER setting up the page geometry
\usepackage{sectsty}
\allsectionsfont{\sffamily\mdseries\upshape} % (See the fntguide.pdf for font help)
\newcommand{\R}{\mathbb{R}}

\newcommand{\mM}{\mathcal{M}}

\newcommand{\bV}{\mathbf{V}}

\newcommand{\bh}{h}
\newcommand{\bx}{x}
\newcommand{\by}{y}

\newcommand{\bP}{{\mathbf{P}}}
\newcommand{\bS}{{\mathbf{S}}}
\newcommand{\bI}{{\mathbf{I}}}
\newcommand{\bA}{{\mathbf{A}}}
\newcommand{\bB}{{\mathbf{B}}}
\newcommand{\bN}{{\mathbf{N}}}

\newcommand{\VEC}{{\mathrm{vec}}}

\newcommand{\abs}[1]{{\left\lvert #1 \right\rvert}}

\newcommand{\changed}[1]{{#1}}
\newcommand{\changedtwo}[1]{{#1}}
\newcommand{\ignore}[1]{{}}

\DeclareMathOperator{\rank}{rank}

\DeclareMathOperator*{\argmin}{argmin}
\DeclareMathOperator{\col}{col}
\DeclareMathOperator{\row}{row}

\newtheorem*{proposition*}{Proposition}
\newtheorem{proposition}{Proposition}[section]
\newtheorem{lemma}[proposition]{Lemma}
\newtheorem{theorem}[proposition]{Theorem}
\theoremstyle{remark}
\newtheorem{remark}[proposition]{Remark}
\newtheorem*{remark*}{Remark}

\numberwithin{equation}{section}

\title{Alternating least squares as moving subspace correction 
}
\author{Ivan V. Oseledets\thanks{Skolkovo Institute of Science and Technology, Skolkovo Innovation Center, 143026 Moscow, Russia (\texttt{i.oseldets@skoltech.ru})} \and Maxim V. Rakhuba\thanks{Seminar for Applied Mathematics, ETH Zurich, R\"amistrasse 101, 8092 Zurich, Switzerland (\texttt{maksim.rakhuba@sam.math.ethz.ch})} \thanks{Work by MR on this project was performed while he was a junior research scientist at Skolkovo Institute of Science and Technology} \and Andr\'e Uschmajew\thanks{Hausdorff Center for Mathematics \& Institute for Numerical Simulation, University of Bonn, 53115 Bonn, Germany. Current address:  Max Planck Institute for Mathematics in the Sciences, 04103
Leipzig, Germany (\texttt{uschmajew@mis.mpg.de}).
}
}
\date{} % Activate to display a given date or no date (if empty),
\begin{document}
\maketitle

\vspace*{-2ex}

\begin{abstract}
In this note we take a new look at the local convergence of alternating optimization methods for low-rank matrices and tensors. Our abstract interpretation as sequential optimization on moving subspaces yields insightful reformulations of some known convergence conditions that focus on the interplay between the contractivity of classical multiplicative Schwarz methods with overlapping subspaces and the curvature of low-rank matrix and tensor manifolds. While the verification of the abstract conditions in concrete scenarios remains open in most cases, we are able to provide an alternative and conceptually simple derivation of the \changed{asymptotic convergence rate} of the two-sided block power method of numerical algebra for computing the dominant singular subspaces of a rectangular matrix. This method is equivalent to an alternating least squares method applied to a distance function. %{\color{red} Some implications regarding ALS for low-rank tensor approximation in tensor train format are also given.} 
The theoretical results are illustrated and validated by numerical experiments.

\medskip

\end{abstract}

\section{Introduction}

Consider a real-valued function $F(x)$, where $x = (\xi_1,\dots,\xi_N)$ is a tuple of vectors $\xi_i \in \R^{n_i}$. The \emph{alternating optimization} (AO) or \emph{block coordinate descent} methods try to solve the problem
\[
\min F(\bx) = \min F(\xi_1,\dots,\xi_N)
\]
by alternating between updates of single (block) variables $\xi_i$ while fixing all the other $\xi_j$, $j \neq i$:
\[
\xi_i \leftarrow \argmin_{\xi \in \R^{n_i}} F(\xi_1,\dots,\xi_{i-1},\xi,\xi_{i+1},\dots,\xi_N).
\]
In other words, such an update is a minimization of $F$ on the affine linear manifold
\(
\bx +  T_i
\)
with the linear subspaces
\begin{equation}\label{eq: subspaces of BCD method}
T_i =  \{0\} \times \dots \times \{ 0\} \times \R^{n_i} \times \{0\} \times \dots \times \{0\} .
\end{equation}
\changed{If $F$ is smooth enough, the minimization in every substep may be replaced by finding a critical point of $F$ on $x + T_i$. The method is then also known under the name \emph{nonlinear (block) Gauss--Seidel method}.}

Such an approach is effective if optimization on the hyperplanes $\bx +  T_i$ is easy, for instance, because it is of lower dimension, or because $F$ takes a simple form on it. Obviously, the hyperplanes $\bx +  T_i$ are changing during this process, as they depend on $\bx$. \changed{The subspaces $T_i$, however, do not change with $x$. Based on this, the linearization of this method around a critical point of $F$ corresponds to a classical Gauss--Seidel (successive relaxation) method applied to a quadratic model of $F$ and, hence its local convergence to a critical point can be shown under suitable assumptions on the Hessian in the critical point; see~\cite{OR1970} or sec.~\ref{sec: locally constant subspaces}.}

There are too many areas of application of AO to mention here. In this paper, we wish to focus on multilinear optimization. This includes low-rank matrix and tensor approximation. Here the scenario is slightly more structured. Let us explain this using the example of low-rank matrix optimization. Assume we are given a function $f \colon \R^{m \times n} \to \R$ on the space of real $m \times n$ matrices, and we wish to minimize it subject to the constraint $\rank(X) \le k$. Then it is natural to use the parametrization $X = U V^T$ with $U \in \R^{m \times k}$, $V \in \R^{n \times k}$, and attempt solving
\[
\min F(U,V) \coloneqq f(U V^T)
\]
via AO between $U$ and $V$:
\begin{equation}\label{eq: AO for matrices}
U \leftarrow \argmin_{\hat U \in \R^{m \times k}} f(\hat U V^T), \qquad V \leftarrow \argmin_{\hat V \in \R^{n \times k}} f(U \hat V^T).
\end{equation}
\changed{
The easiest example to consider is the Euclidean distance function
\[
f(X) = \frac{1}{2} \| X - B \|^2_F = \frac{1}{2} \sum_{i,j} (x_{ij} - b_{ij})^2
\]
to a given matrix $B$. In this case the AO strategy reads
\begin{equation}\label{eq: ALS AO}
U \leftarrow \argmin_{\hat U \in \R^{m \times k}} \frac{1}{2} \| U V^T - B \|_F^2, \qquad V \leftarrow \argmin_{\hat V \in \R^{n \times k}} \frac{1}{2} \| U V^T - B \|_F^2,
\end{equation}
and is called the \emph{alternating least squares} algorithm. It will be discussed in detail in sections~\ref{sec: power method} and~\ref{sec: als}.
}

%Let $XX^+$ and $X^+ X$ denote the orthogononal projections on the column space $\col(X)$ resp. row space $\row(X)$ of a matrix $X$ (here $X^+$ denotes the Moore-Penrose inverse). 
An alternative viewpoint, however, which is the starting point for the present work, is that in terms of the initial function $f$, the AO procedure~\eqref{eq: AO for matrices} amounts to a sequence of optimization problems
\begin{equation}\label{eq: alternaing subspaces for matrices}
X \leftarrow \argmin_{X \in T_1(X)} f(X), \qquad X \leftarrow \argmin_{X \in T_2(X)} f(X),
\end{equation}
on \emph{varying linear subspaces}
\begin{equation}\label{eq: subspace 1}
T_1(X) = \{ Y \in \R^{m \times n} \colon \row(Y) \subseteq \row(X) \},
\end{equation}
respectively,
\begin{equation}\label{eq: subspace 2}
T_2(X) = \{ Y \in \R^{m \times n} \colon \col(Y) \subseteq \col(X) \}.
\end{equation}
Here $\row$ and $\col$ denote the row and column space of a matrix. To be precise, one should emphasize that the update rules~\eqref{eq: AO for matrices} and~\eqref{eq: alternaing subspaces for matrices} are only equivalent as long as all constructed matrices retain full possible rank $k$. Also note that $X \in T_1(X)$ and $X \in T_2(X)$, hence, we can formally see~\eqref{eq: alternaing subspaces for matrices} as minimizations on affine subspaces $X + T_1(X)$ and $X + T_2(X)$ as for the classical AO method. 

The point we wish to make is that the formulation~\eqref{eq: alternaing subspaces for matrices} is a more appropriate viewpoint on the AO method~\eqref{eq: AO for matrices}, since it is intrinsically invariant under different choices of $U$ and $V$ in the bilinear parametrization $X = UV^T$, which, at least formally, is highly nonunique. To see this more clearly, let us compare the two following pseudocodes:

\noindent
\parbox[t]{.45\textwidth}{
\hrule height 0pt width 0pt
\begin{algorithm}[H]
\DontPrintSemicolon
\KwIn{$U_0 \in \R^{m \times k}$, $V_0 \in \R^{n \times k}$.}
\For{$\ell = 0,1,2,\dots$}{
$\displaystyle U_{\ell + 1} \coloneqq \argmin_{\hat U \in \R^{m \times k}} f(\hat U V_\ell^T)$\;
$\displaystyle V_{\ell+1} \coloneqq \argmin_{\hat V \in \R^{n \times k}} f(U_{\ell+1} \hat V^T$)
}
\caption{Low-rank AO, vanilla}\label{Alg:ALS1}
\end{algorithm}
}
\hfill
\parbox[t]{.48\textwidth}{
\hrule height 0pt width 0pt
\begin{algorithm}[H]
\DontPrintSemicolon
\KwIn{$U_0 \in \R^{m \times k}$, $V_0 \in \R^{n \times k}$.}
\For{$\ell = 0,1,2,\dots$}{
$\displaystyle U \leftarrow \argmin_{\hat U \in \R^{m \times k}} f(\hat U V_\ell^T), \quad U = Q_1R_1$\;
$\displaystyle V \leftarrow \argmin_{\hat V \in \R^{n \times k}} f(Q_1 \hat V^T), \quad V = Q_2 R_2$\;
$U_{\ell + 1} \coloneqq Q_1 R_2^{T}, \quad V_{\ell+1} \coloneqq Q_2$
}
\caption{Low-rank AO with QR}\label{Alg:ALS2}
\end{algorithm}
}

\medskip

\noindent The algorithm on the right uses QR decompositions of the factors $U$ and $V$ in order to keep the low-rank representations stable, which is generally advised in practice \changed{as it keeps the $\argmin$ problems better conditioned. This is easily seen for the least squares problems~\eqref{eq: ALS AO} that arise for the squared Frobenius distance: if, e.g., the output matrix $V$ of a previous step is badly conditioned, then the linear operator $U \mapsto U V^T$ is comparably badly conditioned, and the exact solution $U = B V (V^T V)^{-1}$ of the next least squares problem $\min_U \| U V^T - B \|_F$ may be difficult to compute accurately since the matrix $V^T V$ needs to be inverted (assuming it is invertible). If, on the other hand, $V = QR$ is first replaced by its q-factor, $\hat V = Q$, then the next update is just $\hat U = B \hat V$. However, it is easy to check that $U = B Q R^{-T}$ and, therefore, $U V^T = B Q Q^T = \hat U {\hat V}^T$. In other words, in both cases the same matrix is computed, but the second strategy does not require matrix inversion, just the computation of a QR decomposition.}

At first, Algorithm 2 appears considerably harder to analyze than Algorithm 1, which is a plain AO method. A closer inspection, however, reveals that this is not true in the case that the solutions to the minimizations problems are unique (for instance, if all matrices retain rank $k$ and $f$ is strictly convex), since then in both algorithms the same sequences of low-rank matrices $X_\ell = U_\ell V_\ell^T$ are constructed when starting from the same initialization. The underlying reason is that replacing $U$ by its QR-factor $Q$ does not change the column space, and replacing $V$ by its QR-factor does not change the row space. Hence the subspaces of $\R^{m \times n}$ over which the $\argmin$s are taken are the same in both algorithms. The superiority of the `subspace viewpoint' compared to the `representation viewpoint' lies in realizing this \changed{theoretical} equivalence of both algorithms, although numerically they may still behave quite differently.

The above example of low-rank matrix optimization via AO generalizes to the scenario where we are given a multilinear map
\[
\tau \colon V_1 \times \dots \times V_d \to \bV
\]
mapping from $d$ linear spaces $V_1,\dots,V_d$ to a space $\mathbf{V}$, and wish to optimize a function
\begin{equation}\label{eq: multilinear composition}
F(\xi_1,\dots,\xi_d) = f(\tau(\xi_1,\dots,\xi_d)).
\end{equation}
For instance the tasks of computing approximations to tensors in low-rank canonical polyadic (CP), tensor train, or (hierarchical) Tucker formats are of this type; see~\cite{U2012,RU2013,M2013,EHK2015}.

The aim of this paper is to subsume previous local convergence analysis of AO for multilinear optimization~\cite{U2012,RU2013} (see sec.~\ref{sec: conclusion} for an overview on related work) into a transparent theorem that reduces to the subspaces correction method for the linearized problem at a fixed point. Furthermore, in sec.~\ref{sec: block power method} we apply our framework to derive in a new way the (known) convergence rate of a two-sided block power method for computing the dominant $k$-dimensional singular subspaces of a matrix, by relating this power method to AO for the distance function.

\changed{Unfortunately, our techniques are currently not yet in a shape that would allow for substantially new insights in the analysis of alternating least squares for low-rank approximation of higher-order tensors like, e.g., for local convergence analysis of the \emph{higher-order power method} (AO for best rank-one approximation). One reason is that we are lacking an analogous statement to Theorem~\ref{th: convergence rate for identity Hessian} that relates the local contractivity to the spectral radius of a single product of operators. Yet we believe our general setup of AO with moving subspaces will be useful for the tensor case as well in the future. Some references to known results on AO for tensors are given at the end of the paper.} 

\section{Abstract setup}\label{sec: abstract setup}

To generalize our two motivating examples, we consider a $C^1$ function $f \colon \bV \to \bV$ on a Hilbert space $\bV$. To every $\bx \in \bV$ we attach a closed subspace $T(\bx)$ of $\bV$. Further, we assume that we are given a possibly overlapping partition
\[
 T(\bx) = T_1(\bx) + \dots + T_d(\bx)
\]
into $d$ closed subspaces $T_i(\bx)$. Then we define $d$ maps
\[
\bP_i \colon \bV \to \mathcal{L}(\bV),  \quad i=1,\dots,d,
\]
such that for every $\bx \in \bV$ the linear operator $\bP_i(\bx)$ is the orthogonal projection onto the space $T_i(\bx)$. Correspondingly, we let $\bP(\bx)$ be the orthogonal projector on $T(\bx)$.

Next, let $\bS_i$, $i=1,\dots,d$, be (nonlinear) operators on $\bV$ such that $\by = \bS_i(\bx)$ satisfies
%\begin{itemize}
%\item[\upshape (i)]
\begin{equation}\label{eq: definition of S}
\by \in \bx + T_i(\bx), \qquad \bP_i(\bx) \nabla f(\by) = 0.
\end{equation}
%\end{itemize}
It means that $\bS_i$ maps $\bx$ to a relative critical point of $f$ on the hyperplane $\bx + T_i(\bx)$. If, for instance, $f$ is strictly convex and coercive, than such an operator $\bS_i$ is uniquely defined and corresponds to minimizing $f$ on $\bx + T_i(\bx)$. 
	    
AO on moving hyperplanes corresponds to an iteration of the form
\begin{equation}\label{eq: fp iteration}
 \bx_{\ell+1} = \bS(\bx_\ell) \coloneqq (\bS_d \circ \dots \circ \bS_1)(\bx_\ell).
\end{equation}
In the following we consider points $\bar x \in \bV$ which are fixed points of every $\bS_i$, that is,
\[
\bar{\bx} = \bS_i(\bar{\bx}), \quad i =1,\dots,d. 
\]
Then $\bar x$ is obviously a fixed point of $\bS$. \changed{We further note that if $\bar x$ is a fixed point of every $\bS_i$, then $\nabla f(\bar x)$ is orthogonal to $T(\bar x)$. The converse is also true under mild assumptions, for instance, by ensuring that on every hyperplane $x + T_i(x)$ there exists only one point $y$ satisfying~\eqref{eq: definition of S} (e.g., $f$ is strictly convex with bounded sublevel sets), or, if this is not the case, by requiring that $y$~\eqref{eq: definition of S} should be chosen as close as possible to $x$.}

\changed{
In this work, we wish to investigate the local convergence properties of the fixed point iteration~\eqref{eq: fp iteration} under the assumption that all $\bP_i$, all $\bS_i$, and also $\bP$ are continuously (Fr\'echet) differentiable mappings in a neighborhood of $\bar x$. 
Without going into detail, we mention that for optimization tasks in low-rank tensor formats as mentioned in the context of Eq.~\eqref{eq: multilinear composition} such smoothness assumptions are typically ensured if the fixed point has maximal feasible rank. For instance, local analysis of AO for low-rank matrices $UV^T$ as considered in~\eqref{eq: AO for matrices} will require the factors $U$ and $V$ to have full column rank $k$; cf. sec.~\ref{sec: block power method}.}

\changed{
The local contractivity around $\bar x$ is governed by the spectral properties of the derivative $\bS'(\bar x)$. By the chain rule,
\begin{equation}\label{eq: chain rule}
\bS'(\bar x) = \prod_{i=d}^1 \bS'_i(\bar x).
\end{equation}
The derivatives $\bS_i'(\bar{\bx})$ are computed in the next section.} Some preliminary properties, however, are obtained by differentiating the equation
\[
\bP_i(\bx)(\bS_i(\bx) - \bx) = \bS_i(\bx) - \bx.
\]
It gives the relation
\begin{equation*}\label{eq: derivation1}
\bP_i'(\bx;\bh) (\bS_i^{}(\bx) - \bx) + \bP_i^{}(\bx)(\bS_i'(\bx) \bh - \bh) = \bS_i'(\bx)\bh - \bh
\end{equation*}
for all $\bh \in \bV$. Here, $\bP_i'(\bx;\bh) \in \mathcal{L}(\bV)$ denotes the application of the derivative of $\bP_i(\bx)$ at $\bx$ to $\bh$. Hence, in a fixed-point $\bar{\bx} = \bS_i(\bar{\bx})$, it holds
\begin{equation}\label{eq: projection to orthogonal complement}
\changed{\bS_i'(\bar x)h = \bP_i^{}(\bar x) \bS_i'(\bar x)h + (\bI - \bP_i^{}(\bar x))h.}
%(\bI - \bP_i(\bar{\bx})) \bS_i'(\bar{\bx}) \bh = (\bI - \bP_i(\bar{\bx})) \bh.
\end{equation}
This equation is interesting as it shows the following.

\begin{proposition}\label{prop: initial properties of S'}
Assume $\bP_i$ and $\bS_i$ are continuously differentiable around a fixed point $\bar{\bx} = \bS(\bar{\bx})$.
\begin{itemize}
\item[\upshape (i)]
The subspaces $T_i(\bar{\bx})$ \changed{and $T(\bar x)$ are both} invariant subspaces of $\bS_i'(\bar{\bx})$.
\item[\upshape (ii)]
The restriction of $\bS_i'(\bar{\bx})$ to the orthogonal complement $T_i(\bar{\bx})^\bot$ has \changed{all its singular values bounded from below by one, and equals the identity on $T_i(\bar{\bx})^\bot$ if and only if $T_i(\bar{\bx})^\bot$ is also an invariant subspace of $\bS_i'(\bar{\bx})$.}
\changed{
\item[\upshape (iii)] The subspace $T(\bar{\bx})$ is an invariant subspace of $\bS'(\bar{\bx})$, that is, it holds
\[
\bS'(\bar x) \bP(\bar x) = \bP(\bar x) \bS'(\bar x) \bP(\bar x).
\]
The restriction of $\bS'(\bar{\bx})$ to the orthogonal complement $T(\bar{\bx})^\bot$ has \changed{all its singular values bounded from below by one, and equals the identity on $T(\bar{\bx})^\bot$ if and only if $T(\bar{\bx})^\bot$ is also an invariant subspace of $\bS'(\bar{\bx})$.}
}
\end{itemize}
\end{proposition}

\changed{
\begin{proof}
Ad (i). Obviously, by~\eqref{eq: projection to orthogonal complement}, $h \in T_i(\bar x)$ is mapped to $T_i(\bar x)$. On the other hand, since $T_i(\bar x)$ is a subspace of $T(\bar x)$, the element $(\bI - \bP_i^{}(\bar x))h$ in~\eqref{eq: projection to orthogonal complement} belongs to $T(\bar x)$ for every $h \in T(\bar x)$. Hence $T(\bar x)$ is also an invariant subspace of $\bS_i'(\bar x)$.

Ad (ii). Equation~\eqref{eq: projection to orthogonal complement} shows that $\bS_i'(\bar x) h = h$ for all $h \in T_i(\bar{\bx})^\bot$ if and only if $\bP_i^{}(\bar x) \bS_i'(\bar x)h = 0$ for such $h$, which is equivalent to $T_i(\bar{\bx})^\bot$ being an invariant subspace of $\bS_i'(\bar x)$. In any case, it holds, by orthogonality of both terms in~\eqref{eq: projection to orthogonal complement}, that $\| \bS_i'(\bar x) h \|^2 \ge \| h \|^2$ for all $h \in T_i'(\bar x)^\bot$, which shows that the singular values of the restriction to that space cannot be smaller than one.

Ad (iii). Since $T(\bar x)$ is an invariant subspace of every $\bS_i'(\bar x)$ by (i), it follows from the chain rule~\eqref{eq: chain rule} that it is also an invariant subspace of $\bS'(\bar x)$. Even more, an induction based on~\eqref{eq: projection to orthogonal complement} shows that
\[
(\bI - \bP(\bar x))\bS'(\bar x) = (\bI - \bP(\bar x))(\bI - \bP_d(\bar x)) \cdots (\bI - \bP_1(\bar x)).
\]
So, since the $T_i(\bar x)$ are subspaces of $T(\bar x)$, it holds that
\[
(\bI - \bP(\bar x))\bS'(\bar x)(\bI - \bP(\bar x)) = \bI - \bP(\bar x).
\]
This implies the assertion as in (ii).
\end{proof}
}

\changed{
The proposition shows that we can only hope for contractivity of the map $\bS'(\bar{\bx})$ on its invariant subspace $T(\bar{\bx})$. Therefore, in what follows, we focus on the spectral radius of $\bS'(\bar{\bx}) \bP (\bar x) = \bP(\bar x) \bS'(\bar{\bx}) \bP (\bar x)$. Concerning the convergence of our  fixed point iteration $x_{\ell + 1} = \bS(x_\ell)$, the contractivity of $\bS'(\bar{\bx}) \bP (\bar x)$ turns out to be sufficient for local linear convergence in two notable cases: (i) the subspace $T(\bar x)$ equals the whole space $\bV$ as is the case for classical AO with the subspaces given by~\eqref{eq: subspaces of BCD method}; or (ii) the iterates $(x_\ell)$ lie on a smooth submanifold $\mathcal{M}$ and $T(\bar x)$ is the tangent space to that manifold at $\bar x$. This scenario is often encountered in low-rank matrix and tensor optimization via AO, when the objective function~\eqref{eq: multilinear composition} is considered and the image of the multilinear map $\tau$ is locally a manifold. We will demonstrate this for low-rank matrix approximation in sec.~\ref{sec: block power method}.
}

\subsection{Computation of derivatives}

\changed{We recall that $f \colon \bV \to \R$ is said to be twice continuously (Fr\'echet) differentiable in a neighborhood of $\bar x \in \bV$, if for every $x$ in that neighborhood there exists a bounded linear form $f'(x)$ on $\bV$ and a bounded bilinear form $f''(x)$ on $\bV \times \bV$, which both depend continuously on $x$, such that
\[
f(\bx + \bh) = f(\bx) + f'(x) h + \frac{1}{2} f''(x)(h,h) + o(\| h \|^2).
\]
The bilinear forms $f''(x)$ are necessarily symmetric; see, e.g.,~\cite[section~(8.12.2)]{Dieudonne1969}. Hence, since $\bV$ is a Hilbert space, there exist elements $\nabla f(\bx) \in \bV$ (gradient) and unique bounded self-adjoint operators $\bA(\bx)$ (Hessian) on $\bV$, both depending continuously on $x$, such that
\[
f(\bx + \bh) = f(\bx) + \langle \nabla f(\bx),\bh \rangle + \frac{1}{2} \langle \bh, \bA(\bx) \bh \rangle + o(\| h \|^2)
\]
for all $x$ in a neighborhood of $\bar x$. Note that $x \mapsto \bA(x)$ is the (Fr\'echet) derivative of the map $x \mapsto \nabla f(x)$.} 
For brevity the following shorthand notation will be used for the rest of the paper:
\[
\bar{\bP}_i \coloneqq \bP_i(\bar{\bx}), \qquad \bar{\bP} \coloneqq \bP(\bar{\bx}), \qquad  \bar{\bA} \coloneqq \bA(\bar{\bx}), \qquad  \bar{\bB}_i \coloneqq (\bar{\bP}_i\bar{\bA}\bar{\bP}_i)^{-1}.
\]
\changed{The inverse operator ${\bB}_i$ is obtained by considering $\bar{\bP}_i\bar{\bA}\bar{\bP}_i$ as an operator on $T_i(\bar \bx)$.}

To obtain a formula for $\bS'(\bar{\bx})$, we differentiate each $\bS_i$ separately. The derivatives $\bS_i'(\bar \bx)$ are given as follows.

\begin{proposition}\label{prop: derivative of S}
Assume that $\bP_i$ and $\bS_i$ are continuously differentiable in a neighborhood of a fixed point $\bar{\bx}$, and that $f$ is twice continuously differentiable around $\bar{\bx}$. \changed{Then $\bP_i'(\bar{\bx};\bh) \nabla f(\bar{\bx}) \in T_i(\bar{\bx})$.}
%Let $\bar{\bA} = \nabla^2 f(\bar{\bx})$ be the Hessian at $\bar{\bx}$, and $\bar{\bP}_i = \bP_i(\bar{\bx})$.
If the linear operator $\bar{\bP}_i \bar{\bA}\bar{\bP}_i$ is invertible on $T_i(\bar{\bx})$, then%, with $\bar{\bB}_i = (\bar{\bP}_i\bar{\bA}\bar{\bP}_i)^{-1}$,
\begin{align}
\bS_i'(\bar{\bx})\bh &= \bh - \bar{\bB}_i^{} \bar{\bP}_i^{} \bar{\bA}\bh - \bar{\bB}_i^{} \bP_i'(\bar{\bx};\bh) \nabla f(\bar{\bx}) \label{eq: S' 1}.
\end{align}
In particular,
\[
\bS_i'(\bar{\bx}) = -\bar{\bB}_i^{} \bP_i'(\bar{\bx};\bh) \nabla f(\bar{\bx}) \quad \text{on $T_i(\bar{\bx})$.}
\]
\end{proposition}

%Note that it follows from~\eqref{eq: derivation1} and~\eqref{eq: projection to orthogonal complement}, that for any $\bh \in \bV$ the linear operator $\bP_i'(\bar{\bx};\bh)$ maps into the space $T_i(\bar{\bx})$. Hence the composition of $\bar{\bB}_i$ with this operator is well defined.

\begin{proof}
Differentiating the equation $\bP_i(\bx) \nabla f(\bS_i(\bx)) = 0$ yields
\begin{equation}\label{eq: derivation 1}
\bP_i'(\bx;\bh) \nabla f(\changedtwo{\bS_i(\bx)}) + \bP_i^{}(\bx) \bA(\bx) \bS_i'(\bx) \bh = 0
\end{equation}
for all variations $\bh \in \bV$. 
%, where we use
%\(
%\bA(\bx) = \nabla^2 f(\bx)
%\)
%to denote the Hessian of $f$ at $\bx$.
Splitting the term of interest $\bS_i'(\bx)\bh$ in~\eqref{eq: derivation 1} into its parts on $T_i(\bar{\bx})$ and the orthogonal complement, we get
\[
\bP_i^{}(\bx)\bA(\bx)\bP_i^{}(\bx) \bS_i'(\bx) \bh = - \bP_i^{}(\bx)\bA(\bx)(\bI - \bP_i^{}(\bx))\bS_i'(\bx)\bh  - \bP_i'(\bx;\bh) \nabla f(\changedtwo{\bS_i(\bx)}).
\]
At a fixed point, we can use~\eqref{eq: projection to orthogonal complement}. Therefore
\[
\bar{\bP}_i^{}\bar{\bA}\bar{\bP}_i^{} \bS_i'(\bar{\bx}) \bh = - \bar{\bP}_i^{}\bar{\bA}(\bI - \bar{\bP}_i^{})\bh - \bP_i'(\bar{\bx};\bh) \nabla f(\bar{\bx}).
\]
\changed{This equation shows that $\bP_i'(\bar{\bx};\bh) \nabla f(\bar{\bx})$ lies in $T_i(\bar x)$.} Assuming further that $\bar{\bP}_i\bar{\bA}\bar{\bP}_i$ has an inverse $\bar{\bB}_i$ on $T_i(\bar{\bx})$, we get
\[
 \bar{\bP}_i^{} \bS_i'(\bar{\bx}) \bh = \bar{\bP}_i^{}\bh - \bar{\bB}_i^{} \bar{\bP}_i^{} \bar{\bA}\bh - \bar{\bB}_i^{} \bP_i'(\bar{\bx};\bh) \nabla f(\bar{\bx}).
\]
Using~\eqref{eq: projection to orthogonal complement} once more, one arrives at
\begin{align*}
 \bS_i'(\bar{\bx}) \bh &= (\bI - \bar{\bP}_i^{})\bS_i'(\bar{\bx})\bh + \bar{\bP}_i^{} \bS_i'(\bar{\bx}) \bh 
 = (\bI - \bar{\bP}_i^{})\bh + \bar{\bP}_i^{}\bh - \bar{\bB}_i^{} \bar{\bP}_i^{} \bar{\bA}\bh - \bar{\bB}_i^{} \bP_i'(\bar{\bx};\bh) \nabla f(\bar{\bx}),
\end{align*}
which is~\eqref{eq: S' 1}. %To get~\eqref{eq: \bS' 2}, we note that $\bar{\bB}\bar{\bP}\bar{\bA}\bar{\bP} \bh = \bar{\bP} \bh$, and $\bP'(\bar{\bx};\bar{\bP} \bh) = 0$ since $P$ is constant on the hyperplane $E(\bar{\bx})$.
\end{proof}

It will be useful to simplify notation. We denote
\begin{equation}\label{eq: A-orthogonal projections}
\bar{\bP}_i^{\bar{\bA}} \coloneqq \bar{\bB}_i^{} \bar{\bP}_i^{} \bar{\bA}.
\end{equation}
If $\bar{\bA}$ is a positive definite operator, then $\bar{\bB}_i$ is always well defined and $\bar{\bP}_i^{\bar{\bA}}$ allows an interpretation as the $\bar{\bA}$-orthogonal projection onto subspace $T_i(\bar{\bx})$, that is, an orthogonal projection with respect to the inner product $(\bx,\by) \mapsto \langle \bx, \bar{\bA} \by \rangle$.\footnote{To see it, observe (we omit the subscript $i$) that $\bar{\bB} \bar{\bP} = \bar{\bP} \bar{\bB} \bar{\bP}$ is self-adjoint and therefore $\langle \bar{\bP}^{\bar{\bA}} \bx, \bar{\bA} (\bI - \bar{\bP}^{\bar{\bA}})\bx \rangle = \langle \bar{\bB} \bar{\bP} \bar{\bA} x, (\bar{\bA} - \bar{\bA} \bar{\bB} \bar{\bP} \bar{\bA})x \rangle = \langle   x, (\bar{\bA}\bar{\bB} \bar{\bP}\bar{\bA} - \bar{\bA}\bar{\bB} \bar{\bP}\bar{\bA} \bar{\bB} \bar{\bP} \bar{\bA})x \rangle = 0$ for all $\bx \in \bV$, since $\bar{\bP} \bar{\bA} \bar{\bB} \bar{\bP} = \bar{\bP} \bar{\bA} \bar{\bP} \bar{\bB} \bar{\bP} = \bar{\bP}$. Hence $\bar{\bP}^{\bar{\bA}}\bx$ is $\bar{\bA}$-orthogonal to $(\bI - \bar{\bP}^{\bar{\bA}})\bx$.}

Further, we define the linear operator $\bar{\bN}_i$ on $\bV$ such that
\begin{equation}\label{eq: operator N}
\bar{\bN}_i^{} \bh \coloneqq \bP_i'(\bar{\bx};\bh) \nabla f(\bar{\bx})
\end{equation}
for all $\bh$. With this notation, and under the assumptions of Proposition~\ref{prop: derivative of S}, $\bS_i'(\bar{\bx})$ can be conveniently written as
\[
\bS_i'(\bar{\bx}) = (\bI - \bar{\bP}_i^{\bar{\bA}}) - \bar{\bB}_i\bar{\bN}_i.
\]

The formula for $\bS'(\bar x)$ is now obtained by the chain rule. For later reference we formulate it as a theorem.

\begin{theorem}\label{thm: derivative of S}
Assume that all $\bP_i$ and $\bS_i$ are continuously differentiable in a neighborhood of a fixed point $\bar{\bx}$, and that $f$ is twice continuously differentiable around $\bar{\bx}$.  
Assume all $\bar{\bB}_i = (\bar{\bP}_i \bar{\bA}\bar{\bP}_i)^{-1}$ exist on $T_i(\bar{\bx})$. Then
\begin{equation}\label{eq: derivative of S}
\bS'(\bar \bx) = \prod_{i=d}^1 \bS_i'(\bar \bx) = \prod_{i=d}^1 [ (\bI - \bar{\bP}_i^{\bar{\bA}}) - \bar{\bB}_i\bar{\bN}_i].
\end{equation}
\end{theorem}

\subsection{Curvature free cases ($\bar{\bN}_i = 0$)}\label{sec: model problems}

An easy case to investigate is when all $\bar{\bN}_i = 0$, since in this case we obtain the formula
\[
\bS'(\bar \bx) = \prod_{i=d}^1 (\bI - \bar{\bP}_i^{\bar{\bA}}),
\]
which is well known from the theory of subspace correction methods for the solution of linear systems, specifically the \emph{multiplicative Schwarz method}. The following statement is obtained from the standard results on the multiplicative Schwarz method \changed{(see, e.g.,~\cite[Theorem~4.2]{XuZikatanov2002} for the Hilbert space case), by restricting everything to the subspace $T(\bar{x})$ and considering the equivalent $\bar \bA$-inner-product $(\bx,\by)_{\bar \bA} = (\bx, \bar \bA \by )$. We recall that all subspaces $T_i(x)$ and $T(x)$ have been assumed to be closed, which is important; cf. Theorem~4.6 in~\cite{XuZikatanov2002}.}

\begin{theorem}\label{thm: linear convergence curvature free}
Assume all $\bar{\bN}_i = 0$ and $\bar{\bA}$ is positive definite on $T(\bar \bx)$.\footnote{It means that there exists $m > 0$ such that $\langle \bx, \bar \bA \bx \rangle \ge m \| \bx \|^2$ for all $\bx \in T(\bar \bx)$.} Then $\| \bar \bP \bS'(\bar x) \bar \bP \|_{\bar \bA} < 1$. In particular, $\rho(\bS'(\bar x) \bar \bP) = \rho(\bar \bP \bS'(\bar x) \bar \bP) < 1$. 
\end{theorem}

The case $\bar{\bN}_i = 0$ considered here arises in two notable cases.

\subsubsection{Locally constant subspaces}\label{sec: locally constant subspaces}

If the subspaces $T_i(\bx)$ are the same for all $\bx$ in a neighborhood of $\bar \bx$, then $\bP_i'(\bar \bx) = 0$. This case occurs in the classical nonlinear Gauss--Seidel method discussed in the introduction, \changed{where the subspaces $T_i$ are fixed and do not depend on $x$ at all.} Hence, in this case Theorem~\ref{thm: linear convergence curvature free} simply recovers the well-known fact that the local convergence rate of the nonlinear (block) Gauss-Seidel method equals the rate of the linear block Gauss-Seidel method with the Hessian as the system matrix; cf., e.g.,~\cite{OR1970}.

\subsubsection{Zero gradient}\label{sec: zero gradient}

The operators $\bar{\bN}_i$ are also zero in fixed points satisfying $\nabla f(\bar \bx) = 0$. 
An interesting scenario for this situation is low-rank optimization where an unconstrained critical point lies on a considered manifold of low-rank matrices or tensors. This scenario is presented in section~\ref{sec: block power method}; see, in particular, Lemma~\ref{lem: derivatives of projections}.

\subsection{A nontrivial example including curvature ($\bar{\bN}_i \neq 0$)}

A case with $\bar{\bN}_i \neq 0$, but allowing for considerable simplification, is obtained for $d=2$ when $T(\bar \bx) = T_1(\bar \bx) + T_2(\bar \bx)$ can be decomposed into its intersection and two other $\bar{\bA}$-orthogonal parts. This case occurs for problems of low-rank best approximation. In these cases, $f$ is a quadratic function with Hessian equal to identity matrix: $\bar{\bA} = \bI$; see sec.~\ref{sec: block power method} below.

\begin{theorem}\label{th: convergence rate for identity Hessian}
 In addition to the assumptions of Theorem~\ref{thm: derivative of S}, suppose the following two conditions hold:
 \begin{itemize}
  \item[\upshape (i)] $\bar \bP_1^{\bar \bA}$ and $\bar \bP_2^{\bar \bA}$ commute,\footnote{This condition is equivalent to the fact that the $\bar \bA$-orthogonal projector on $T(\bar x)$ allows the two decompositions
\(
\bar \bP^{\bar \bA} = \bar \bP_1^{\bar \bA} + \bar \bP_2^{\bar \bA} - \bar \bP_2^{\bar \bA} \bar \bP_1^{\bar \bA}
\)
and
\(
\bar \bP^{\bar \bA} = \bar \bP_1^{\bar \bA} + \bar \bP_2^{\bar \bA} - \bar \bP_1^{\bar \bA} \bar \bP_2^{\bar \bA}.
\)}
  \item[\upshape (ii)] $\bar{\bN}_i = 0$ on $T_i(\bar \bx)$ for $i = 1,2$.
 \end{itemize}
Then
\[
  \rho(\bS'(\bar \bx) \bar{\bP}) = \rho(\bar{\bB}_2 \bar{\bN}_2 \bar{\bB}_1 \bar{\bN}_1 \bar{\bP}).
\]
\end{theorem}
\begin{proof}
When $\bar \bP_1^{\bar \bA}$ and $\bar \bP_2^{\bar \bA}$ commute, it is easily verified that the operator $(\bI - \bar{\bP}_1^{\bar{\bA}}) \bar \bP$ maps to $T_2(\bar \bx)$. By~\eqref{eq: derivative of S} and assumption (ii), it then holds that
 \[
   \bS'(\bar \bx) \bar \bP = -(\bI - \bar{\bP}_2^{\bar{\bA}}  - \bar{\bB}_2 \bar{\bN}_2) \bar{\bB}_1 \bar{\bN}_1 \bar \bP = -(\bI - \bar{\bP}_2^{\bar{\bA}}  - \bar{\bB}_2 \bar{\bN}_2) \bar \bP \bar{\bB}_1 \bar{\bN}_1 \bar \bP.
 \]
It is a well-known fact that the spectral radius of the product of two operators is invariant under the order of factors. Thus, by the above formula, the spectral radius of $\bS'(\bar \bx) \bar{\bP}$ is the same as the spectral radius of $-\bar{\bB}_1 \bar{\bN}_1 \bar \bP (\bI - \bar{\bP}_2^{\bar{\bA}}  - \bar{\bB}_2 \bar{\bN}_2) \bar \bP = \bar{\bB}_1 \bar{\bN}_1 \bar \bP \bar{\bB}_2 \bar{\bN}_2 \bar \bP$. Here we have used that $(\bI - \bar{\bP}_2^{\bar{\bA}}) \bar \bP$ maps to $T_1(\bar \bx)$ by (i). Changing the order of factors again, we obtain the result.
\end{proof}

\begin{remark}\label{rem: quadratic convergence}
An even stronger result is obviously obtained when again $\bar{\bN}_i = 0$ on the whole space $T(\bar \bx)$ as in sec.~\ref{sec: model problems}. Then $\bS'(\bar \bx) \bar \bP = 0$ and we expect superlinear convergence (given sufficient smoothness of $\bS$). This happens, for instance, when $\nabla f(\bar \bx) = 0$. If, additionally, $f$ is quadratic, then the sequential solution of $\bP_i(\bar \bx) \nabla f(y) = 0$ on $T_i(\bar \bx)$ provides a critical point on the whole space $T(\bar \bx)$ after only one sweep through $i=1,2$ (due to orthogonal residuals). Of course, the condition (i) in Theorem~\ref{th: convergence rate for identity Hessian} is very strong when $\bar{\bA}$ is not the identity operator, as it implies that we are given a possibly overlapping, but otherwise $\bar{\bA}$-orthogonal splitting of the space $T(\bar \bx)$. 
\end{remark}

\section{AO for low-rank matrices}\label{sec: block power method}

We return to the AO method~\eqref{eq: AO for matrices} for solving the problem
\begin{equation}\label{eq: nonsmooth low rank}
\min_{\rank(X) \le k} f(X)
\end{equation}
for a function $f \colon \R^{m \times n} \to \R$, as outlined in the introduction. We first give an overview oh what the abstract setup developed above looks like in this case. We then deal with the alternating least squares method for quadratic functions $f$, and its relation to power iterations in the case that the Hessian of the function is the identity operator. {\changed{By $\langle X, Y \rangle_{F} = \sum_{i,j} x_{ij} y_{ij}$ we denote the Frobenius inner product of two matrices, and by $\| X \|_F$ the corresponding induced norm.}

Starting from an initial guess $X_0 = U_0^{} V_0^T$ of rank $k$, the method produces a sequence $X_{\ell}^{} = U_{\ell}^{} V_\ell^T$ of matrices of rank at most $k$ by minimizing the function $f(UV^T)$ with respect to $U$ and $V$ only in an alternating manner. As long as the matrices $U_\ell$ and $V_\ell$ remain of rank $k$, this method is equivalently described as AO on the varying subspaces defined in~\eqref{eq: subspace 1} and~\eqref{eq: subspace 2}.\footnote{When the rank drops, some formal subtleties appear. In the alternating subspace method the rank can only decrease, but never increase again, whereas in the AO method for $U$ and $V$ the size of the blocks is not changed, and even if, say, $U$ with rank less than $k$ is fixed, the minimizer for $V$ is then not unique and a full-rank matrix $V$ could be selected.} Using the projections
\begin{equation}\label{eq: low-rank projections}
Z \mapsto \bP_1(X)[Z] \coloneqq Z X^+X, \qquad Z \mapsto \bP_2(X)[Z] \coloneqq XX^+Z,\footnote{Note that different from previous notation in $\R^n$ we now use square brackets in $\bP(X)[Z]$ to describe the linear action of $\bP(X)$ on $Z$ in order to avoid confusion with matrix multiplication.}
\end{equation}
these subspaces can also be written as 
\begin{equation*}
\begin{gathered}
T_1(X) = \{ Y \in \R^{m \times n} \colon Y = \bP_1(X)[Y] \},\\ T_2(X) = \{ Y \in \R^{m \times n} \colon Y = \bP_2(X)[Y] \}.
\end{gathered}
\end{equation*} 
Here we recall that the Moore-Penrose inverse of $X \in \R^{m \times n}$ is defined as $X^+ = V \Sigma^{-1} U^T \in \R^{n \times m}$, where $X = U \Sigma V^T$ is a `slim' singular value decomposition of $X$ involving only the positive singular values. It is then obvious that the $\bP_i(X)$ are projections whose ranges are the subspaces $T_i(X)$ as defined in~\eqref{eq: subspace 1}. We also see that $XX^+$ and $X^+X$ are themselves orthogonal projections in $\R^m$ and $\R^n$, respectively. Since the Frobenius inner product of two matrices can be computed column- or row-wise, it then easily follows that the $\bP_i(X)$ are in fact orthogonal projections with respect to this inner product.

It is well known that for every $k$ the set
\[
\mathcal{M}_k = \{ X \in \R^{m \times n} \colon \rank(X) = k \}
\]
is a smooth embedded submanifold of $\R^{m \times n}$ of codimension $(m-k)(n-k)$~\cite[Example~8.14]{Lee2003}. It can further be shown that the space
\[
T(X) = T_1(X) + T_2(X)
\]
is the tangent space to that manifold at $X \in \mathcal{M}_k$. %\footnote{{\color{red} The space $T(X)$ is the image of the linearization of the map $(U,V) \mapsto U V^T$ at $(U,V)$ such that $X =UV^T$ and is hence contained in the tangent space. Since it has the correct dimension, it equals the tangent space.}}
Therefore, if $\bar X$ has rank $k$ and is a fixed point of a (locally) smooth map $\bS \colon \R^{m \times n} \to \R^{m \times n}$ satisfying
\begin{equation}\label{eq: rank condition on S}
\rank(\bS(X)) \le \rank(X)
\end{equation}
for all $X$, then the condition $\rho(\bS'(\bar X) \bar{\bP}) < 1$ is sufficient for R-linear convergence
\begin{equation}\label{eq: R-linear rate via spectral radius}
\limsup_{\ell \to \infty} \| X_\ell - \bar X \|^{1/\ell} \le \rho(\bS'(\bar X) \bar{\bP})
\end{equation}
(in any norm, since we are now in a finite-dimensional setting) of an iteration
\[
X_{\ell + 1} = \bS(X_\ell)
\]
with starting guess $X_0$ of rank $k$ close enough to $\bar X$.\footnote{Let us prove this. If $\bar X$ is a fixed point, then, by continuity, $\bS(X)$ is close to $\bar X$ when $X$ is close to $\bar X$. Hence under the given assumptions, $\rank(\bS(X)) = k$ for all $X$ with $\rank(X) = k$ that are close enough to $\bar X$ (by semicontinuity of rank). Therefore, $\bS$ can be locally regarded as a map between smooth submanifolds of $\mathcal{M}_k$, $\bS'(\bar X)$ maps the tangent space $T(\bar X)$ into itself, and the sufficiency of the condition $\rho(\bS'(\bar X)) < 1$ on $T(\bar X)$ for local contractivity follows in the same way as in linear space using differential calculus on manifolds.} 

From~\eqref{eq: low-rank projections} it is obvious that $\bP_1(X)$ and $\bP_2(X)$ commute. Correspondingly,
\begin{equation}\label{eq: full projection}
\bP(X) = \bP_1(X) + \bP_2(X) - \bP_1(X) \bP_2(X) = \bP_1(X) + \bP_2(X) - \bP_2(X) \bP_1(X)
\end{equation}
is the orthogonal projection on $T(X)$.

\subsection{Derivatives of projections}\label{sec: derivatives of projections}

The reader will have noticed that the mappings $X \mapsto \bP_i(X)$ as defined in~\eqref{eq: low-rank projections} are not differentiable on $\R^{m \times n}$ unless $\rank(X) = \min(m,n)$, since the map $X \mapsto X^+$ is not. \changed{To resolve this potential conflict to the theory developed above, we can formally extend the projections to smooth maps in a neighborhood of $\mathcal{M}_k$. Indeed, let $\mathcal{D}$ be the open set of all matrices whose $k$-th singular value is strictly larger than the $(k+1)$-th one (such a matrix necessarily has rank at least $k$). To every $X \in \mathcal{D}$ we attach the orthogonal projections $U(X)U(X)^T$ and $V(X)V(X)^T$ onto the subspaces spanned by the dominant left, respectively, right singular vectors gathered as columns in the matrices $U(X) \in \R^{m \times k}$, respectively, $V(X) \in \R^{n \times k}$. These maps are smooth on $\mathcal{D}$. The projections $\bP_i$ from~\eqref{eq: low-rank projections} can hence be extended to smooth maps on $\mathcal{D}$ via
\begin{equation}\label{eq: low-rank projections extended}
\begin{gathered}
\bP_1(X)[Z] = Z V(X)V(X)^T, \qquad \bP_2(X)[Z] = U(X)U(X)^T Z,
\end{gathered}
\end{equation}
which coincide with~\eqref{eq: low-rank projections} if $X \in \mathcal{M}_k$. The formula~\eqref{eq: full projection} remains valid.

We now present the derivatives of $\bP_1$ and $\bP_2$ at points $X \in \mathcal{M}_k$.
We first consider directional derivatives $\bP_i'(X;H)$ for $H \in T(X)$, which is the tangent space to $\mathcal{M}_k$ at $X$. 
The Moore-Penrose pseudoinverse is a smooth map on manifolds of constant rank, and its Riemannian derivative at $X \in \mathcal{M}_k$ is given by
\begin{equation}\label{eq: derivative of pseudo inverse}
\mathrm{D}X^+[H] = - X^+ H X^+ + X^+ (X^+)^T H^T (I - XX^+) + (I - X^+ X) H^T (X^+)^T X^+
\end{equation}
with $H \in T(X)$; see~\cite{GP1973}.
Hence, we compute from~\eqref{eq: low-rank projections} that
\begin{equation}\label{eq: derivative 1}
\begin{aligned}
\bP_1'(X;H)[Z] &=  Z X^+ H + Z \cdot \mathrm{D}X^+[H] \cdot X \\ &= Z X^+ H - ZX^+ H X^+ X + Z(I - X^+ X) H^T (X^+)^T 
\end{aligned}
\end{equation}
for $H \in T(X)$. Here we have used $(I - XX^+) X = 0$ and $(X^+)^T X^+ X = (X^+)^T$. Correspondingly,
\begin{equation}\label{eq: derivative 2}
\begin{aligned}
\bP_2'(X;H)[Z] &= H X^+ Z + X \cdot \mathrm{D}X^+[H] \cdot Z \\
&= H X^+ Z  - X X^+ H X^+ Z + (X^+)^T H^T (I - XX^+) \changedtwo{Z}
\end{aligned}
\end{equation}
for $H \in T(X)$, since $X(I - X^+ X) = 0$ and $X X^+ (X^+)^T = (X^+)^T$.

Regarding directional derivatives $H \in T(X)^\bot$ (which will not be needed later on), we invite the reader to verify that $T(X)^\bot$ consists of all matrices of the form $H = (I - XX^+)E(I - X^+X)$ for some $E$, and that small perturbations of $X$ along such directions do not change the  dominant singular vectors. Hence $t \mapsto \bP_i(X + tH)$ as defined in~\eqref{eq: low-rank projections extended} is constant for small $\abs{t}$ and so
\[
\bP_1'(X;H) = \bP_2'(X;H) = 0 \quad \text{for $H \in T(X)^\bot$}.
\]
Note that the formulas~\eqref{eq: derivative 1} and~\eqref{eq: derivative 2} also yield zero when applied to $H \in T(X)^\bot$ (since $X^+H = 0$ and $HX^+ = 0$) and can hence be used in general.
}

The formulas~\eqref{eq: derivative 1} and~\eqref{eq: derivative 2} can be considerably simplified when $Z$ is orthogonal to the tangent space $T(X)$, since in this case $\bP_1(X)[Z] = 0$, implying $Z X^+ = 0$, and $\bP_2(X)[Z] = 0$, implying $X^+ Z = 0$. For such $Z$,~\eqref{eq: derivative 1} and~\eqref{eq: derivative 2} become
\[
\bP_1'(X;H)[Z] = Z H^T (X^+)^T
\]
and
\[
\bP_2'(X;H)[Z] = (X^+)^T H^T Z.
\]
Note that since we need to derive the operators $\bar{\bN}_i [H] = \bP_i'(\bar{X};H) [\nabla f(\bar{X})]$ defined in~\eqref{eq: operator N} at critical points $\bar{X}$ of $f$ on $\mathcal{M}_k$, where $\nabla f(\bar{X})$ is orthogonal to $T(\bar{X})$, this is indeed the case of interest. For reference we state this as a lemma.
\begin{lemma}\label{lem: derivatives of projections}
Let $X \in \mathcal{M}_k$ be a critical point of $f$ on $\mathcal{M}_k$, that is, $\bP(\bar{X}) \nabla f(\bar{X}) = 0$. Then for the projections~\eqref{eq: low-rank projections} it holds
\begin{equation*}\label{eq: formula N1}
\bar{\bN}_1 [H] \coloneqq \bP_1'(\bar{X};H) \nabla f(\bar{X}) = \nabla f({\bar{X}}) H^T (\bar{X}^+)^T
\end{equation*}
and
\begin{equation*}\label{eq: formula N2}
\bar{\bN}_2 [H] \coloneqq \bP_2'(\bar{X};H)[\nabla f(\bar{X})] = (\bar{X}^+)^T H^T \nabla f(\bar{X})
\end{equation*}
for all $H \in T(\bar{X})$. In particular, $\bar{\bN}_i = 0$ on $T_i (\bar X)$.
\end{lemma}

\begin{remark}
Regarding the initial problem~\eqref{eq: nonsmooth low rank} on $\mathcal{M}_{\le k}$, we remark that the ``smoothness'' assumption $\rank(\bar X) = k$, which has been crucial in the above derivations, is plausible in most applications, except for very special or artificial cases. It has been shown in~\cite[Corollary~3.4]{SU2015} that critical points $\bar X$ of~\eqref{eq: nonsmooth low rank}, for example, local minima, satisfy either $\rank(\bar X) = k$ or $\nabla f(\bar X) = 0$.
\end{remark}

\subsection{Alternating least squares algorithm}\label{sec: als}

When $f$ is a strictly convex quadratic function, the outlined method is known as the \emph{alternating least squares} (ALS) method. Let us give formulas for this important special case in more detail.

For simplicity, we assume that $f(0) = 0$. Then $f$ takes the form
\begin{equation}\label{eq: quadratic cost function}
f(X) = \frac{1}{2} \langle X, \bA[X] \rangle_F - \langle X, B \rangle_F, 	
\end{equation}
where $\bA$ is a symmetric positive definite linear operator on $\R^{m \times n}$, and $B \in \R^{m \times n}$. We have $\nabla f(X) = \bA[X] - B$, and the Hessian at every point is the operator $\bA$.

\changed{
Minimizing the function $f$ without constraint is equivalent to solving the linear matrix equation $\bA[X] = B$. Let
\[
X^* = \bA^{-1}[B]
\]
be the solution. Introducing the $\bA$-norm
\[
\| X \|_\bA = \sqrt{\langle X, \bA[X] \rangle_F}
\]
on $\R^{m \times n}$, we can rewrite the function $f$ as
\[
f(X) = \frac{1}{2} \| X - X^*\|_{\bA}^2 - \frac{1}{2} \| X^* \|_\bA^2.
\]
This shows that minimizing $f$ on $\mathcal{M}_{\le k}$ is equivalent to finding the best rank-$k$ approximation(s) of the true solution $X^*$ in $\bA$-norm, and they serve as approximate low-rank solutions to the linear equation. The ALS algorithm tries to find such minima of $f$ on $\mathcal{M}_{\le k}$.
}

At a given iterate $X_\ell \in \mathcal{M}_{\le k}$, the first step of ALS computes
\[
\bS_1(X_\ell) = \argmin_{X \in T_1(X_\ell)} f(X).
\]
Since $\bA$ is positive definite, there is indeed a unique solution, and it is given as
\begin{equation}\label{eq: S1 for ALS}
\bS_1(X) = (\bP_1(X) \bA \bP_1(X))^{-1}[\bP_1(X)[B]].
\end{equation}
Here, as usual, $(\bP_1(X) \bA \bP_1(X))^{-1}$ is understood as the inverse of the operator $\bP_1(X) \bA \bP_1(X)$ on its invariant subspace $T_1(X)$. The map $X \mapsto \bS_1(X)$ is differentiable on the manifold $\mM_k$ of rank-$k$ matrices, since $\bP_1$ is.

If $\bS_1(X_\ell)$ has rank $k$,\footnote{If not, there are several options, but we ignore that case.} then the next step of ALS computes
\begin{equation}\label{eq: general ALS iteration}
X_{\ell +1} = \bS(X_\ell) \coloneqq \bS_2(\bS_1(X_\ell)) = \argmin_{X \in T_2(\bS_1(X_\ell))} f(X).
\end{equation}
The solution map $\bS_2$ is given as
\begin{equation}\label{eq: S2 for ALS}
\bS_2(X) = (\bP_2(X) \bA \bP_2(X))^{-1}[\bP_2(X)[B]].
\end{equation}
We repeat once more that $X \in T_1(X)$ and $X \in T_2(X)$ for every $X$, so we are in the abstract framework developed in sec.~\ref{sec: abstract setup}.

The original idea of AO for low-rank optimization is to operate on a (nonunique) factorization $X_\ell^{} = U_\ell^{} V_\ell^T$. In terms of these factors, more precisely, their vectorizations, the ALS method becomes the algorithm displayed as Algorithm~\ref{Alg:ALS}, where $\bA$ is to be understood as an $mn \times mn$ matrix and $\otimes$ is the standard Kronecker product for matrices. As explained in the introduction, the QR decompositions are not mandatory in theory, but highly recommended in practice for numerical stability.

\begin{algorithm}[H]
\DontPrintSemicolon
\KwIn{$B \in \R^{m \times n}$, %$U_0 \in \R^{m \times k}$,
$V_0 \in \R^{n \times k}$.}
\For{$\ell = 0,1,2,\dots$}{
$\VEC (U) = \left((V_\ell^T\otimes I) \, \bA \, (V_\ell \otimes I)\right)^{-1} \VEC (BV_\ell),\quad QR = \mathrm{qr}(U)$\;
$U \coloneqq Q$\;
$\VEC (V^T) = \left((I\otimes U^T) \, \bA \, (I\otimes U)\right)^{-1} \VEC (U^TB),\quad QR = \mathrm{qr}(V)$\;
$V_{\ell+1} \coloneqq Q, \quad U_{\ell+1} \coloneqq U R^{T}$\;
%{\color{red} To be checked}
}
\caption{ALS algorithm for~\eqref{eq: quadratic cost function}}\label{Alg:ALS}
\end{algorithm}

\subsection{SVD block power method}\label{sec: power method}

As a special case, we now consider the quadratic function~\eqref{eq: quadratic cost function} with $\bA = \bI$. It corresponds to the task
\begin{equation}\label{eq: best approximation problem}
\min_{\rank(X) \le k} \frac{1}{2}\| X - B \|_F^2
\end{equation}
of computing a best rank-$k$ approximation to matrix $B$ in the Frobenius norm. Since $\bA = \bI$, we have $(\bP_i(X) \bA \bP_i(X))^{-1} \bP_i(X) = \bP_i(X)$ for $i=1,2$, and hence the update formulas~\eqref{eq: S1 for ALS} and~\eqref{eq: S2 for ALS} for AO simplify to
\begin{equation}\label{eq: S1 and S2 for best approximation}
\bS_1(X) = \bP_1(X)[B] =  B X^+ X, \qquad
%\]
%and
%\[
\bS_2(X) = \bP_2(X)[B] = XX^+ B.
\end{equation}
The resulting ALS iteration $X_{\ell+1} = \bS(X_\ell)$ becomes
\begin{equation}\label{eq: Block power method}
X^{}_{\ell + 1/2} = B X^+_\ell X^{}_\ell, \qquad X^{}_{\ell + 1} = X^{}_{\ell + 1/2} X_{\ell + 1/2}^+ B.
\end{equation}

Writing $X_\ell^{} = U_\ell^{} \Sigma_\ell^{} V_\ell^T$, it is easily seen (and shown below) that the sequence generated by~\eqref{eq: Block power method} is the same as in the \emph{simultaneous orthogonal iteration}, which is a two-sided block power method for computing the dominant $k$ left and right singular subspaces of $B$, displayed as Algorithm~\ref{Alg:block power method} (provided $X_0$ has the row space spanned by $V_0$).

\bigskip

\begin{algorithm}[H]
\DontPrintSemicolon
\KwIn{$B \in \R^{m \times n}$, $k \le \min(m,n)$, $V_0 \in \R^{n \times k}$ such that $V_0^T V_0^{} = I$}
\For{$\ell = 0,1,2,\dots$}{
$QR = \mathrm{qr} (BV_\ell)$\tcp*{tall QR decomposition}
$U_{\ell + 1} \coloneqq Q$\;
$QR = \mathrm{qr} (B^T U_{\ell+1})$\tcp*{tall QR decomposition}
$V_{\ell + 1} \coloneqq Q$, \quad $S_{\ell+1} = R^T$\;
}
\caption{Simultaneous orthogonal iteration}\label{Alg:block power method}
\end{algorithm}

\bigskip

Let
\begin{equation}\label{eq: svd of B}
B = \sum_{i=1}^{\min(m,n)} \sigma_i^{} \bar u_i^{} \bar v_i^T
\end{equation}
be the SVD of $B$, that is, $\bar u_1 , \bar u_2, \dots $ and $\bar v_1 ,\bar v_2 , \dots$ are orthonormal systems in $\R^m$ and $\R^n$, respectively, and $\sigma _1 \ge \sigma_2 \ge \dots \ge 0$. If $\sigma_k > \sigma_{k+1}$, then it can be shown that the sequence $X_\ell = U_\ell^{} S_\ell^{} V_\ell^T$ generated in Algorithm~\ref{Alg:block power method} converges to the unique best rank-$k$ approximation
\begin{equation}\label{eq: svd of X}
\bar X = \sum_{i=1}^k \sigma_i^{} \bar u_i^{} \bar v_i^T
\end{equation}
for almost every starting guess $X_0 = U_0^{} V_0^T$. 
In fact, the method produces the same subspaces as the corresponding orthogonal iterations for the symmetric matrices $B^TB$ and $BB^T$, respectively, whose eigenvalues are the $\sigma_i^2$ (zero may be a further eigenvalue). Hence, by well-known results, $U_\ell \to [\bar u_1, \dots, \bar u_k]$ and $V_\ell \to [\bar v_1, \dots, \bar v_k]$ in terms of subspaces for almost every initialization, with a convergence rate $O\left( \sigma_{k+1}^2 / \sigma_k^2 \right)$; see~\cite{B1957,GVL2013}. As an application of our abstract framework \changedtwo{we are able to obtain this (known) rate of convergence from the local convergence analysis of the ALS sequence~\eqref{eq: Block power method}}.

\begin{theorem}\label{thm: block power method}
Let $B$ have singular values $\sigma_1 \ge \dots \ge \sigma_k > \sigma_{k+1} \ge  \dots$, and the unique best rank-$k$ approximation $\bar X$. Then the sequence $X_{\ell+1} = \bS(X_\ell) = \bS_2(\bS_1(X_\ell))$ defined via~\eqref{eq: S1 and S2 for best approximation}, respectively,~\eqref{eq: Block power method} (AO for problem~\eqref{eq: best approximation problem}) is, in exact arithmetic, identical to the sequence $X_\ell = U_\ell^{} S_\ell^{} V_\ell^T$ generated by the simultaneous orthogonal iteration (Algorithm~\ref{Alg:block power method}). With $\bar \bP$ = $\bP(\bar X)$ as before, it holds that
\begin{equation}\label{eq: exact spectral radius}
 \rho(\bS'(\bar X) \bar{\bP}) = \left(\frac{\sigma_{k+1}}{\sigma_k}\right)^{2}.
\end{equation}
Consequently, by~\eqref{eq: R-linear rate via spectral radius}, the sequence $(X_\ell)$ converges (for close enough starting guesses $X_0 \in \mathcal{M}_k$) $R$-linearly to $\bar X$ at a rate
\begin{equation}\label{eq: local convergence sequence}
\limsup_{\ell \to \infty} \| X_{\ell} - \bar X \|^{1/\ell} \le \left(\frac{\sigma_{k+1}}{\sigma_k}\right)^{2}
\end{equation}
(in any norm). The convergence of the column and row spaces can be estimated correspondingly in the sense of the operator norm of projectors as
\begin{equation}\label{eq: local convergence subspace}
\limsup_{\ell \to \infty} \| \bP_i(X_\ell) - \bP_i(\bar X) \|^{1/\ell} \le \left(\frac{\sigma_{k+1}}{\sigma_k}\right)^{2}, \quad i=1,2.
\end{equation}
\end{theorem}

\begin{proof}
We first show by induction that the methods are the same. If $X_\ell$ has the row space spanned by $V_\ell$, then $X_{\ell + 1/2}$ in~\eqref{eq: Block power method} can be written $B V_\ell^{} V_\ell^T$, which has the same column space as $B V_\ell$. Therefore, using $U_{\ell + 1}$ from Algorithm~\ref{Alg:block power method}, we get that $X_{\ell+1}$ from~\eqref{eq: Block power method} equals $U_{\ell + 1}^{} U_{\ell + 1}^T B = U_{\ell + 1}^{} \Sigma_{\ell +1}^{} V_{\ell + 1}^T$.

One may attempt to compute the spectral radius of $\bS'(\bar X) \bar{\bP}$ from the explicit formulas~\eqref{eq: S1 and S2 for best approximation} and~\eqref{eq: derivative of pseudo inverse}, but it will be more elegant to invoke Theorem~\ref{th: convergence rate for identity Hessian}.
Since $\bar{\bA} = \bA = \bI$, the condition in item (i) of that theorem is obviously satisfied ($\bP_1(X)$ and $\bP_2(X)$ commute; see~\eqref{eq: low-rank projections}). The condition (ii), that $\bar{\bN}_i = 0$ on $T_i(\bar X)$, is stated in Lemma~\ref{lem: derivatives of projections}. Taking into account further that the $\bar{\bB}_i$ are identities, Theorem~\ref{th: convergence rate for identity Hessian} yields the formula
\begin{equation}\label{eq: reduction to rhoN1N2}
\rho(\bS'(\bar X) \bar{\bP}) = \rho(\bar{\bN}_2 \bar{\bN}_1 \bar{\bP})
\end{equation}
for the iteration~\eqref{eq: Block power method}. By Lemma~\ref{lem: derivatives of projections},
\begin{equation*}
\bar{\bN}_2 [ \bar{\bN}_1 [ H ] ] = (\bar{X}^+)^T \bar{X}^+ H (\nabla f(\bar{X}))^T \nabla f(\bar{X}) \quad \text{for $H \in T(\bar X)$.} 
\end{equation*}
Taking further into account that $\nabla f(\bar X) = \bar X - B$, this shows that 
\begin{equation}\label{eq: N2N1 on T}
\bar{\bN}_2 \bar{\bN}_1 \bar \bP = \bar \bP \bar{\bN}_2 \bar{\bN}_1 \bar \bP = \bar \bP \cdot [(\bar{X}^+)^T \bar{X}^+  \otimes (\bar X - B) (\bar X - B)^T] \cdot \bar \bP,
\end{equation}
where we use the Kronecker product operator notation (see~\eqref{eq: Kronecker operator}). \changed{
By~\eqref{eq: svd of B} and~\eqref{eq: svd of X}, the rank-one matrices $E_{i,j} = \bar u_i^{} \bar v_j^T$, $i=1,\dots,m$, $j = 1,\dots,n$, form an orthonormal system of eigenvectors of the operator $(\bar{X}^+)^T \bar{X}^+  \otimes (\bar X - B) (\bar X - B)^T$, corresponding to eigenvalues $\lambda_{i,j} = \sigma_j^2 / \sigma_i^2$ for $i \le k$ and $j > k$, and $\lambda_{ij} = 0$ otherwise. \changedtwo{The largest of these eigenvalues is $\lambda_{k,k+1}^{} = \sigma_{k+1}^2 / \sigma_k^2$. Since the corresponding eigenvector $E^{}_{k,k+1} = \bar u_k^{} \bar v_{k+1}^T$ belongs to $T_2(\bar X) \subseteq T(\bar X)$ (see~\eqref{eq: low-rank projections}), the formula~\eqref{eq: N2N1 on T} implies that $\lambda_{k,k+1}^{}$ is also the largest (in modulus) eigenvalue of $\bar{\bN}_2 \bar{\bN}_1 \bar \bP$, which proves the assertion~\eqref{eq: exact spectral radius}.
}
}

Since $\rank(\bar X) = k$, it follows that $\bS$ is a local contraction on the manifold $\mathcal{M}_k$ in the neighborhood of $\bar X$, and the R-linear convergence rate of $\| X_\ell - \bar X \|_F$ is as asserted (see the explanations for~\eqref{eq: R-linear rate via spectral radius} in sec.~\ref{sec: block power method}).

Let us show that~\eqref{eq: local convergence sequence} implies~\eqref{eq: local convergence subspace} for $\bP_1$. 
For all $Z$ with $\| Z \|_F = 1$, we can estimate
\begin{align*}
\| (\bP_1(X_\ell) - \bP_1(\bar X))[Z] \|_F &= \| Z(X_\ell^+ X_\ell^{} - \bar X^+ \bar X) \|_F \\ &\le \| X_\ell^+ X_\ell^{} - \bar X^+ \bar X \|_2 \\
&= \| X_\ell^+ (X_\ell^{} - \bar X) + (X_\ell^+ - \bar X_{}^+) \bar X \|_2 
=O( \| X_\ell^{} - \bar X \|_2), 
\end{align*}
since $X_\ell \to \bar X$ on $\mathcal{M}_k$ (implying that $\| X^+_\ell \|_2$ is bounded).
\end{proof}

\changed{
\begin{remark}\label{rem: remark about one sweep}
When $\sigma_k > \sigma_{k+1} = 0$, that is $\rank (B) = k$, the theorem yields $\rho(\bS'(\bar X) \bar{\bP}) = 0$ which technically indicates superlinear convergence. In fact, this is a situation where Remark~\ref{rem: quadratic convergence} applies: it holds $\nabla f(\bar X) = 0$ and, hence, $\bar \bN_i = 0$.  
However, as it is known, and not difficult to see, the power method~\eqref{eq: Block power method} initialized with the correct rank $k$ will find $\bar X = B$ after only one sweep for almost every starting guess $X_0$. The only condition is that $X_{1/2}^{} = B X_0^+ X_0^{}$ is of rank $k$ which, in particular, is true for all $X_0$ in some neighborhood of $\bar X = B$.
\end{remark}
}

\changed{
\subsection{Kronecker product operators}

A main feature in the previous derivation of the local convergence rate of the block power method via the ALS analysis was the possibility of applying Theorem~\ref{th: convergence rate for identity Hessian} for the computation of the spectral radius $\rho(\bS'(\bar X) \bar{\bP})$, since for $\bar \bA = \bA = \bI$ the $\bar \bA$-orthogonal projectors $\bar \bP_1^{\bar \bA}$ and $\bar \bP_2^{\bar \bA}$ commute. Note that $\bI = I \otimes I$ is a Kronecker product of two identity matrices. To allow for at least a small generalization, we now investigate the case that $\bar \bA = A_1 \otimes A_2$ is a Kronecker product of symmetric positive definite matrices. One can show that in this case the projectors $\bar \bP_1^{\bar \bA}$ and $\bar \bP_2^{\bar \bA}$ still commute and, hence, derive estimates for $\rho(\bS'(\bar X) \bar{\bP})$ for this case based on Theorem~\ref{th: convergence rate for identity Hessian}. There is, however, a simpler way to analyze the ALS method for Kronecker product operators by reducing it to the block power method again.

Consider a quadratic function
\[
f(X) = \frac{1}{2} \langle X, \bA[X] \rangle_F - \langle X, B \rangle_F
\] 
on $\R^{m \times n}$, where the Hessian is a Kronecker product operator,
\[
\bar \bA = \bA = A_2 \otimes A_1, \qquad A_1 \in \R^{m \times m}, \ A_2 \in \R^{n \times n},
\]
by which we mean that
\begin{equation}\label{eq: Kronecker operator}
\bA[X] = (A_2 \otimes A_1)[X] = A_1^{} X A_2^T.
\end{equation}
Since $\bA$ should be symmetric positive definite, we assume that $A_1$ and $A_2$ are both symmetric positive definite.

We have already noted in sec.~\ref{sec: als} that minimizing $f$ subject to $\rank(X) \le k$ corresponds to finding the best rank-$k$ approximation of the global minimum $X^* = \bA^{-1}[B]$ in $\bA$-norm. For the case that $\bA$ is a Kronecker product operator of the considered type, the best rank-$k$ approximation in the $\bA$-norm can be in principle computed directly via SVD. For this we rewrite
\begin{align*}
f(X) &= \frac{1}{2}\| \bA^{1/2}[X] - \bA^{-1/2}[B] \|_F^2 - \frac{1}{2}\| X^* \|_\bA^2 \\ &= \frac{1}{2}\| A_1^{1/2} X A_2^{1/2} - A_1^{-1/2}BA_2^{-1/2}] \|_F^2 - \frac{1}{2}\| X^* \|_\bA^2.
\end{align*}
Since left or right multiplication by an invertible matrix does not change the rank, we can clearly see that the global minima of $f$ on $\mathcal{M}_{\le k}$ are given as
\begin{equation}\label{eq: global minimizer in Kronecker case}
\bar X = A_1^{-1/2} \bar Y A_2^{-1/2},
\end{equation}
where $\bar Y$ is a best rank-$k$ approximation of $A_1^{-1/2}BA_2^{-1/2}$, which can be computed through SVD. Obviously, there is a unique global minimum $\bar X$ if and only if $A_1^{-1/2}BA_2^{-1/2}$ has a unique best rank-$k$ approximation in the Frobenius norm.

It should therefore not come as a surprise that the ALS method in this case will be equivalent to the block power method for the matrix $A_1^{-1/2}BA_2^{-1/2}$. To prove this, it will be convenient to have the ALS update formulas in explicit matrix notation at hand. Using a decomposition
\[
X = U S V^T, \quad U^T U = I_k, \quad V^T V = I_k,
\]
the formulas~\eqref{eq: S1 for ALS} and~\eqref{eq: S2 for ALS} become
\begin{equation}\label{eq: S1 for Kron}
\bS_1(X) = A_1^{-1} B V (V^T A_2 V)^{-1} V^T
\end{equation}
and
\begin{equation}\label{eq: S2 for Kron}
\bS_2(X) = U ( U^T A_1 U)^{-1} U^T B A_2^{-1}.
\end{equation}

\begin{theorem}\label{thm: theorem for Kronecker product}
Let $\bA = A_2 \otimes A_1$ with $A_1$ and $A_2$ being symmetric positive definite, and $B \in \R^{m \times n}$. Denote by $\varsigma_1 \ge \varsigma_2 \ge \dots$ the singular values of $C = A_1^{-1/2}BA_2^{-1/2}$. If then $\varsigma_k > \varsigma_{k+1}$, then for almost every starting point $X_0 \in \mathcal{M}_k$ the sequence $X_{\ell+1} = \bS(X_\ell) = \bS_2(\bS_1(X_\ell))$ defined via~\eqref{eq: S1 for Kron}, respectively,~\eqref{eq: S2 for Kron} (Algorithm~\ref{Alg:ALS}) is well defined and converges to the unique global minimum $\bar X = A_1^{-1/2} \bar Y A_2^{-1/2}$ of the function $f$ given by~\eqref{eq: quadratic cost function} on $\mathcal{M}_{\le k}$, where $\bar Y$ is the unique best rank-$k$ approximation of $C$ in the Frobenius norm. In fact, for almost every $X_0 \in \mathcal{M}_k$ it holds (in exact arithmetic) that
\[
X_{\ell} = A_1^{-1/2} Y^{}_\ell A_2^{-1/2},
\]
where $Y_\ell$ is the sequence generated by the SVD block power method~\eqref{eq: Block power method} (Algorithm~\ref{Alg:block power method}) applied to matrix $C$ with starting point $Y_0^{} = A_1^{1/2} X_0^{} A_2^{1/2}$. The asymptotic $R$-linear convergence rate is estimated as
\[
\limsup_{\ell \to \infty} \| X_{\ell} - \bar X \|^{1/\ell} = \limsup_{\ell \to \infty} \| Y_{\ell} - \bar Y \|^{1/\ell} \le \left(\frac{\varsigma_{k+1}}{\varsigma_k}\right)^{2}
\]
(in any norm).
\end{theorem}

\begin{proof}
We know that under the given assumptions on \changedtwo{$C$} the sequence $Y_\ell$ is well defined (that is, every half-step in~\eqref{eq: Block power method} remains in $\mathcal{M}_k$) for almost every starting point $Y_0 \in \mathcal{M}_k$ and converges to $\bar Y$ at an asymptotic $R$-linear rate $(\varsigma_{k+1} / \varsigma_k)^2$.

Let $X_\ell = A_1^{-1/2} Y_\ell A_2^{-1/2} \in \mathcal{M}_k$ be true for some $\ell$. \changedtwo{Then, by~\eqref{eq: S1 for Kron}, we can write
\[
\bS_1(X_\ell) = A_1^{-1/2} C P  A_2^{-1/2}
\]
with
\[
P_\ell^{} = A_2^{1/2}V (V^T A_2 V)^{-1} V^T A_2^{1/2},
\] 
and the columns of $V$ forming a basis for the row space of $X_\ell$.}
We claim that
\(
P = Y_\ell^+ Y^{}_\ell
\).
To see this we note that $P$ is symmetric and $P^2 = P$. Further, the null space of $P$ obviously equals the orthogonal complement of the column space of $A_2^{1/2} V$. Hence $P$ is the orthogonal projector on this subspace, which, however, equals the row space of $Y_\ell = A_1^{1/2} X_\ell A_2^{1/2}$. This shows $P = Y_\ell^+ Y^{}_\ell$. It follows that
\[
\bS_1(X_\ell) = A_1^{-1/2} C Y_\ell^+ Y_\ell^{} A_2^{-1/2} = A_1^{-1/2} \changedtwo{Y_{\ell+1/2}} A_2^{-1/2},
\]
where \changedtwo{$Y_{\ell+1/2}$ is the next half-step from $Y_\ell$ in the block power method for $C$}. The argument for the second half step is analogous and the proof of the theorem is finished by induction. (Both inequalities in the asserted equality $\limsup \| X_{\ell} - \bar X \|^{1/\ell} = \limsup \| Y_{\ell} - \bar Y \|^{1/\ell}$ are immediate for a submultiplicative matrix norm.)  
\end{proof}

}

\changed{
\begin{remark}
Analogously to Remark~\ref{rem: remark about one sweep} we note that in the case $\varsigma_k > \varsigma_{k+1} = 0$ the theorem technically indicates a superlinear local convergence rate, while in reality the method will actually find the correct solution $\bar X = X^* = A_1^{-1} B A_2^{-1}$ in just one sweep for almost all $X_0$ (and, in particular, for all $X_0$ in a neighborhood of $\bar X$).
\end{remark}
}

\section{Numerical experiments}

The goal of this section is to investigate the agreement between the theoretical estimates and the numerical behaviour. The goal is to minimize the quadratic cost function~\eqref{eq: quadratic cost function} subject to $\rank(X) \le k$ using the ALS Algorithm~\ref{Alg:ALS}. We consider four examples for the Hessian operator $\bA = \bar{\bA}$: the identity operator, a simple Kronecker product operator, a Laplace-like operator, and a random positive definite operator.

In all experiments, the initial guesses $X_0$ in Algorithm~\ref{Alg:ALS} have been chosen randomly. In the figures we depict lines corresponding to the theoretical rate of convergence $\rho(\bS'(\bar X) \bar \bP)$ by the color black, which has been computed numerically at the observed limit point $\bar X$ by forming a matrix representation of the linear operator $\bS'(\bar X) =  [ (\bI - \bar{\bP}_2^{\bar{\bA}}) - \bar{\bB}_2\bar{\bN}_2][ (\bI - \bar{\bP}_1^{\bar{\bA}}) - \bar{\bB}_1\bar{\bN}_1]$ and solving a full eigenvalue problem to find the spectral radius. \changed{To assemble such a matrix representation of $\bS'(\bar X)$, we applied it successively to the (reshaped) columns of an $mn \times mn$ identity matrix using the formulas provided by Lemma~\ref{lem: derivatives of projections}.

In the plots, the theoretical rate $\rho(\bS'(\bar X) \bar \bP)$ is compared with the relative errors $\frac{\| X_\ell - \bar X \|}{\| \bar X \|}$ as well as with the relative norm of projected residuals $\frac{\| \bP(X_\ell)[A[X_\ell] - B] \|}{\| \bP(X_0)[A[X_0] - B] \| }  $, which are the quantities of interest from the perspective of Riemannian optimization (since $A[X_\ell] - B = \nabla f(X_\ell)$). Moreover, the latter have the advantage that they can be monitored in practice during the iteration.
}

\subsection{Case $\bar{\bA} = \bI$ and $\bar{\bA} = A_2 \otimes A_1$}

Consider the ALS method for problem~\eqref{eq: best approximation problem}, that is, minimizing the function
\[
f(X) = \frac 12 \|X - B\|_F^2
\]
subject to $\rank(X) \le k$, where $B\in \mathbb{R}^{n\times n}$ is a given matrix with a predefined distribution of singular values. The goal is to find the best rank-$k$ approximation $\bar X$ of $B$.

Specifically, we consider $n=50$, $k=2$, set $\sigma_2(B) = 10^{-3}$, and test with different $\sigma_3(B)$'s. By Theorem~\ref{thm: block power method}, the ALS method in this case is locally convergent at with an asymptotic $R$-linear rate $(\sigma_{k+1}/\sigma_k)^2$ and, in fact, this bound is sharp.\footnote{Using the classical linear algebra approach related to spectral decomposition and power method one should see that this rate is in fact attained for almost every starting guess.} As illustrated in Figure~\ref{fig: block power} (a), we observe close experimental agreement with this bound. 

\begin{figure*}
	\centering
	\subfloat[]{
	\includegraphics[width=0.5\textwidth]{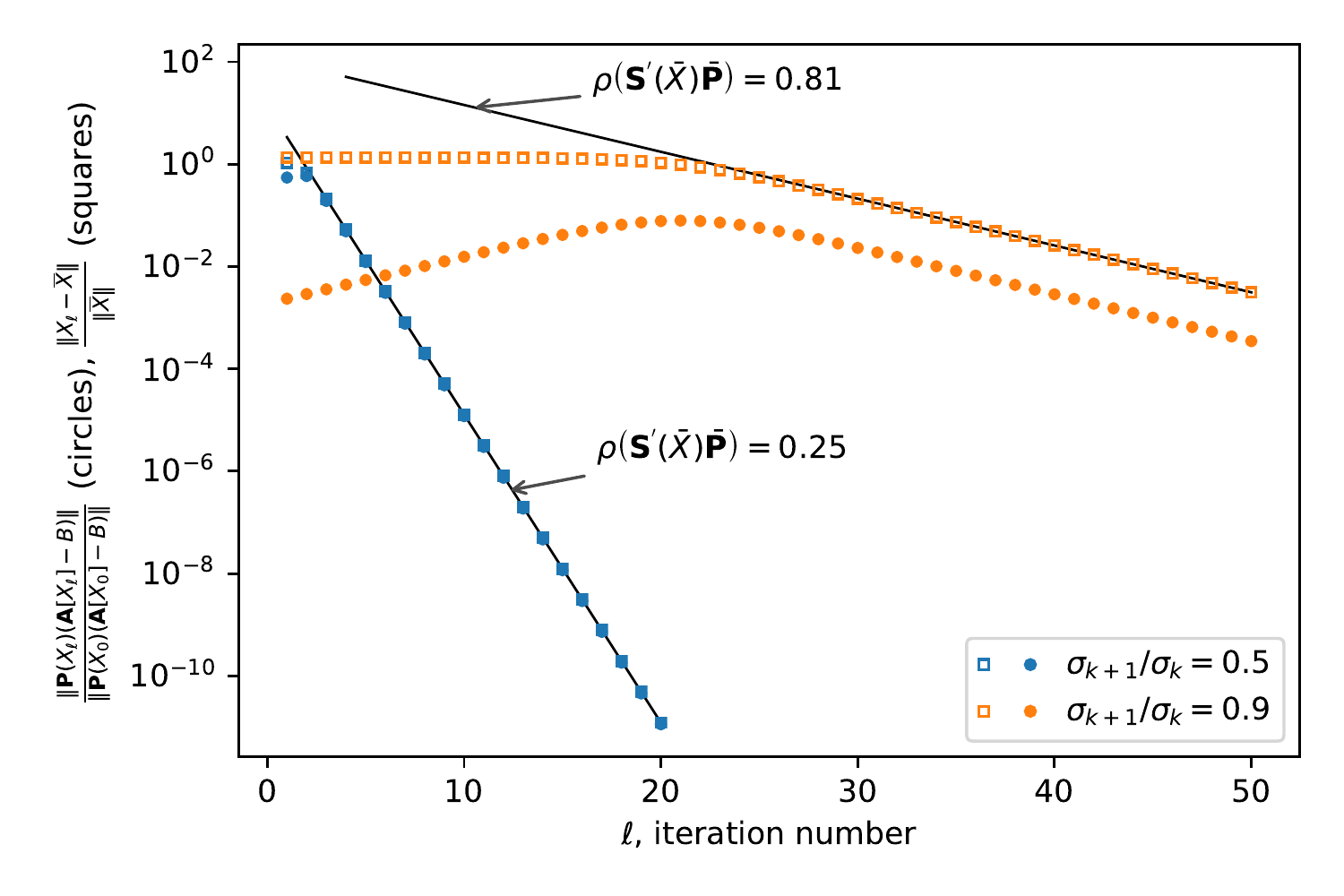}
	}
	\subfloat[]{
	\includegraphics[width=0.5\textwidth]{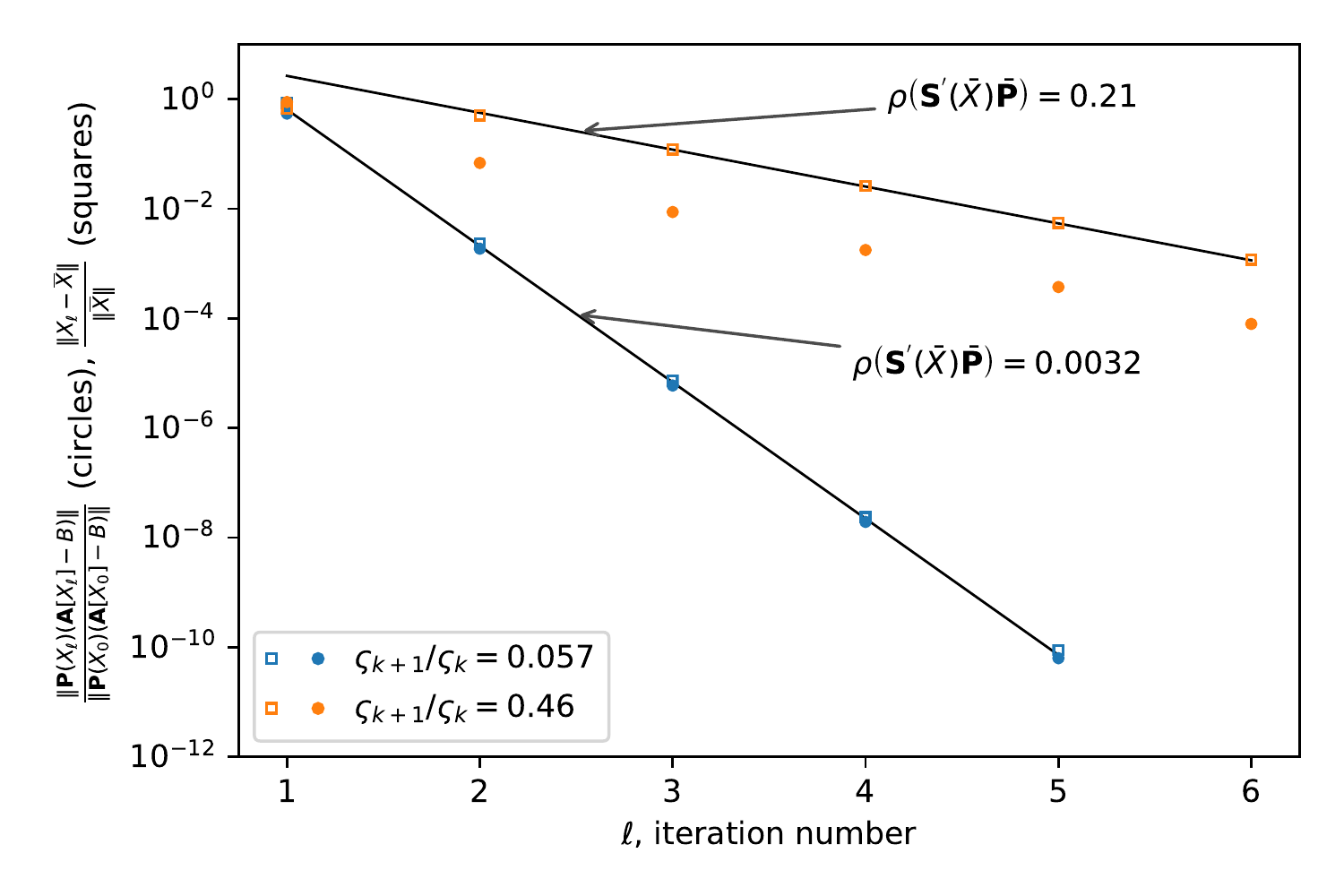}
	}
	\caption{\changed{(a) Relative errors in ALS w.r.t. the iteration number for $\bar{\bA} = \bI$, $k=2$, $\sigma_k=10^{-3}$. Black lines have slopes $\rho(\bS'(\bar X) \bar \bP)=(\sigma_{k+1}/\sigma_k)^2$ for different $\sigma_{k+1}$, while colored dots represent the observed convergence. (b) Same experiment for $\bar{\bA} = A_2 \otimes A_1$ with positive definite $A_1, A_2$. The $\varsigma_k$ are the singular values of $A_1^{-1/2} B A_2^{-1/2}$. The value $\varsigma_k = 10^{-3}$ is kept fixed.}}
	\label{fig: block power}
\end{figure*}

Note that if $\sigma_{k+1}= 0$, then the method (generically) converges in one iteration (so technically superlinear) since row and column spaces of $B$ are found immediately (Remark~\ref{rem: remark about one sweep}). 

\changed{
Note that the other extreme case, when $\sigma_{k+1}=\sigma_k$, is not covered by our local convergence analysis, which does not necessarily mean absence of convergence to \emph{some} best rank-$k$ approximation. However, usually $\bar X$ itself will not be a point of attraction of the block power method for all $X_0$ in the neighborhood. For instance, when $B = \begin{bmatrix} 1 & 0 \\ 0 & 1 \end{bmatrix}$ and $k = 1$, the matrix $\bar X = \frac{1}{2}\begin{bmatrix} 1 & 1 \\ 1 & 1 \end{bmatrix}$ is a best rank-one approximation and a fixed point of the method. However, for $X_0 = \begin{bmatrix} \alpha & \beta \\ \alpha & \beta \end{bmatrix}$, the method becomes stationary after one sweep at $X_1 = \frac{1}{\alpha^2 + \beta^2} \begin{bmatrix} \alpha^2 & \alpha \beta \\ \alpha \beta & \beta^2 \end{bmatrix}$, which is also a best rank-one approximation of $B$, but equals $\bar X$ only when $\alpha = \beta = 1/2$. At the same time, $X_0$ can be arbitrarily close to $\bar X$.
}

\changed{
We also verify our theoretical result on Kronecker product operators. We randomly generate matrices $\tilde{A}_1, \tilde{A}_2 \in \R^{n \times n}$ and use the operator $\bA = \bar{\bA} = A_2 \otimes A_1$ with $A_1 = \tilde{A}_1^{} \tilde{A}_1^T$, $A_2 = \tilde{A}_2^{} \tilde{A}_2^T$ in the function $f$ in~\eqref{eq: quadratic cost function}. The global minimizer $\bar X$ of $f$ on $\mathcal{M}_{\le k}$ is given by~\eqref{eq: global minimizer in Kronecker case} and can be computed using SVD. By Theorem~\ref{thm: theorem for Kronecker product}, the ALS algorithm should converge for almost any $X_0 \in \mathcal{M}_k$ to $\bar X$ and the asymptotic $R$-linear rate is $(\varsigma_{k+1} / \varsigma_k)^2$, where $\varsigma_k$ are the singular values of $A_1^{-1/2} B A_2^{-1/2}$. Figure~\ref{fig: block power} (b) shows perfect agreement with the theoretical prediction. Here again we considered $k=2$, and generated different right hand sides $B$ such that always $\varsigma_2 = 10^{-3}$.
}

\subsection{More general symmetric positive definite $\bar{\bA}$}

We now go beyond Kronecker product operators. First, we consider an entirely random symmetric positive definite matrix
\[
 \bA = \mathbf{R}^\top \mathbf{R} \in \R^{n^2 \times n^2}, 
\]
where $\bold{R}$ is a matrix with each element produced by the standard normal distribution. As another example, we take the highly structured matrix arising in the discretization of a two-dimensional Laplacian on uniform tensor product grid with zero Dirichlet boundary conditions:
\[
\bA = I_n\otimes D_n + D_n \otimes I_n \in \mathbb{R}^{n^2\times n^2}, \quad D_n = (n+1)^2\,\mathrm{tridiag}(-1, 2, -1)_{n\times n}.
\]

Figure~\ref{fig: general A predefined B} displays experimental results for the ALS algorithm with $n=50$ and $k=2$. The matrix $B$ has been chosen such that the solution of the matrix equation $\bA[Y]=B$ (that is, the global minimizer of~\eqref{eq: quadratic cost function} without low-rank constraint) has a predefined distribution of singular values. Similarly to the previous experiments we set $\sigma_2(Y) = 10^{-3}$, while $\sigma_3(Y)$ varies and the goal is to find the best rank-$2$ approximation $\bar X$. As we observe, numerical behaviour is in close agreement with theoretical estimates. While the convergence rate $\rho(\bar \bS'(\bar X) \bar \bP)$ does not equal $\sigma_{k+1} / \sigma_k$ as in the case $\bar \bA = \bI$, it still seems related to this ratio for both choices of $\bar \bA$. Remarkably, $\rho(\bar \bS'(\bar X) \bar \bP)$ is considerably smaller than one even when $\sigma_{k+1}/\sigma_{k}$ is close to one. A decisive question for future work would be for which combinations of $\bar \bA$ and $B$ this can be rigorously shown.

Also note that, in contrast to the case of Kronecker product operators, there is no superlinear convergence in this experiment when $\sigma_{k+1}=0$. In this situation the minimizer of~\eqref{eq: quadratic cost function} on the rank-$k$ variety is the same as the global one, so the curvature-free case considered in sec.~\ref{sec: zero gradient} (zero gradient) applies. Local linear convergence of the ALS method to this minimizer is then guaranteed by Theorem~\ref{thm: linear convergence curvature free}.

\begin{figure*}%[hb]
\centering
\subfloat[]{\includegraphics[width=0.5\textwidth]{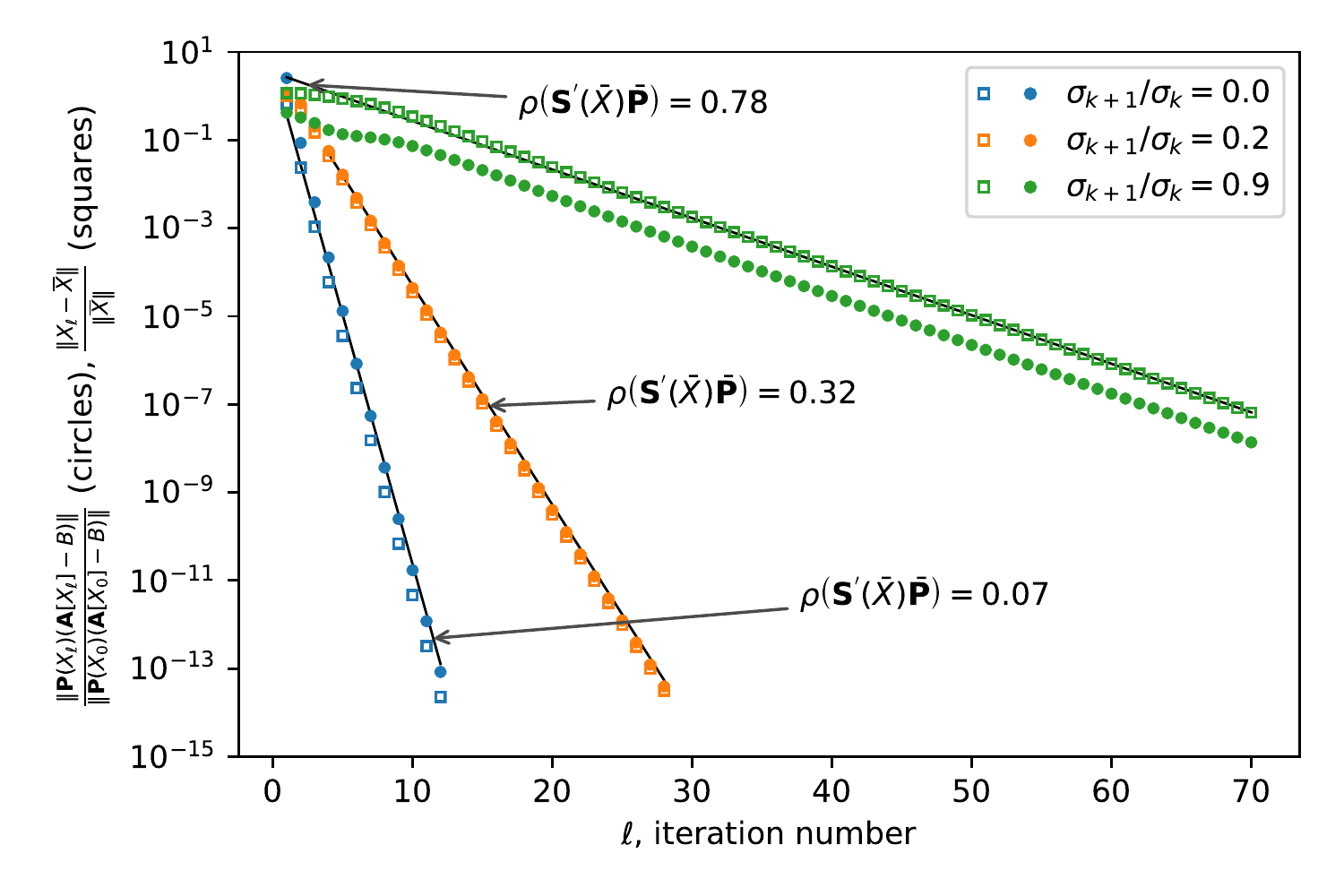}\label{subfig: rand1}
 }
\subfloat[]{\includegraphics[width=0.5\textwidth]{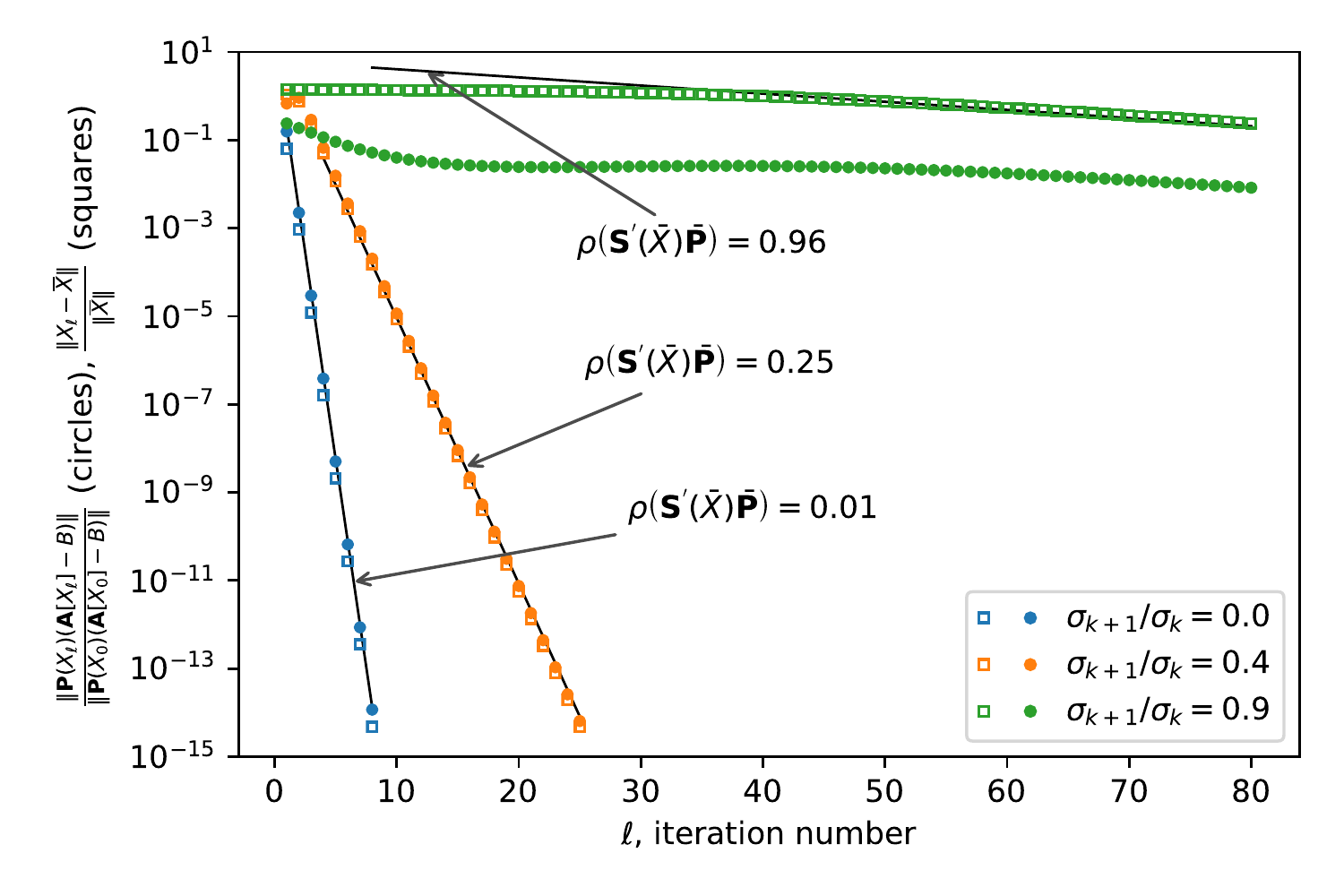}\label{subfig: lapl1}
 }
\caption{Relative errors in ALS w.r.t. the iteration number for $\bA$: \protect\subref{subfig: rand1} random symmetric positive definite matrix and \protect\subref{subfig: lapl1} Laplace matrix. Experiments for $k=2$ and predefined singular value $\sigma_k=10^{-3}$ of the solution of $\bA[X]=B$. Black lines have slopes corresponding to the theoretical convergence rate for different $\sigma_{k+1}$.}
\label{fig: general A predefined B}
\end{figure*}

\section{Conclusion}\label{sec: conclusion}

The goal of this paper was to derive transparent conditions for the local linear convergence of AO algorithms for multilinear and low-rank optimization, specifically the ALS algorithm, which reflect the underlying geometry and do not depend on the representation of low-rank tensors as in previous works. Due to multilinearity of the cost function, single optimization steps take place on linear subspaces, leading (in particular for quadratic cost functions) to an interpretation of AO as a nonlinear subspace correction method (with changing subspaces). Using a sufficiently general framework, a formula for the derivative of the nonlinear iteration function can be obtained (Theorem~\ref{thm: derivative of S}), which displays the interplay of terms from the classic linear subspace correction method with the curvature of the underlying low-rank manifold and the gradient of the cost function in a clear way. The main task remains to show that the spectral radius of this derivative is less than one in applications of interest. This is true in low-rank optimization tasks where the global minimizer lies on the considered low-rank manifold. The case where this is not true is more subtle. For AO for low-rank matrices, the curvature terms can be considerably simplified, which allows for an alternative, analytic proof for the well-known convergence rate of the simultaneous orthogonal iteration for computing the dominant left and right singular subspaces of a matrix. While the main trick (Theorem~\ref{th: convergence rate for identity Hessian}) that was used to obtain this result may not apply in more general situations, we hope that our framework can be a useful starting point in future work for finding rigorous statements for the observed linear convergence of AO and ALS in other applications, like low-rank solutions of Lyapunov equations (cost function~\eqref{eq: quadratic cost function}) and low-rank tensor approximation.

\subsection*{Related work}

In~\cite{U2012} and~\cite{RU2013} the local convergence of the ALS algorithm has been analyzed for low-rank tensor approximation in the CP and tensor train formats, respectively, using the nonlinear Gauss-Seidel approach for a cost function of the form~\eqref{eq: multilinear composition}, e.g., using an explicit representation of low-rank tensors. To address the problem that the Hessian of this cost function cannot be positive definite due to nonuniqueness of tensor representations, equivalence classes of representations (level sets of the function $\tau$ in~\eqref{eq: multilinear composition}) are introduced. Linear convergence is then established for the case that the null space of the Hessian equals the tangent space of the orbit of equivalent representations. The idea is certainly analogous to restricting the operator $\bS'(\bar x)$ to the subspace $T(\bar x)$ as in the present paper, but we believe that our approach provides a much clearer picture by avoiding the unintuitive concept of equivalent representations. A formula
\begin{equation}\label{eq: Hessian for multilinear}
\langle h, \nabla F^2(\xi) h \rangle = \langle \tau'(\xi) h, \nabla^2 f(x) \tau'(\xi) h \rangle + \langle \nabla f(x), \tau''(\xi)[h,h] \rangle
\end{equation}
for the Hessian at $x = \tau(\xi)$ is given in~\cite{RU2013}, which features the Hessian $\nabla^2 f(x)$ on the tangent space of the image of $\tau$, and the interaction of curvature ($\tau''(x)$) and gradient $\nabla f$ as in our work (cf. the definition~\eqref{eq: operator N} of $\bar \bN_i$). In particular, it is concluded that local convergence is guaranteed if $\nabla f(\bar x) = 0$ under an injectivity assumption on $\tau'(\bar \xi)$.

For optimization problems on manifolds, the interplay of global Hessian, gradient, and curvature as displayed in~\eqref{eq: Hessian for multilinear} is gathered in the important concept of the \emph{Riemannian Hessian}. This is thoroughly discussed in~\cite{ATMA2009}; see, in particular, section~6 therein. Similar to its role in smooth optimization in linear spaces, the positivity of the Riemannian Hessian ensures local (Riemannian) convexity and hence contractivity of many Riemannian optimization methods; see the book~\cite{absil}. For manifolds of low-rank matrices, the curvature terms in this Hessian have been obtained in other works. Specifically,~\cite[Proposition 2.2]{V2013} features a formula that makes the Kronecker-type interplay between $\bar X^+$ and $\nabla f(\bar X)$ in the curvature (see Lemma~\ref{lem: derivatives of projections}) clearly visible, albeit for a special case of the cost function~\eqref{eq: quadratic cost function} related to matrix completion (with $\bA = \bP_{\Omega}$ being a projection on given entries $\Omega$). In~\cite{KSV2016}, the curvature term in the Riemannian Hessian is explicitly neglected to derive Riemannian Gauss-Newton-type methods on low-rank tensor manifolds.

In the works~\cite{M2013} and~\cite{EHK2015}, convergence of the ALS method for low-rank tensor approximation has been investigated using alternative techniques. In particular, questions on cluster points and global convergence are addressed. Also, examples for sublinear, linear, and superlinear convergence are presented in~\cite{EHK2015}. The references~\cite{GolubZhang2001,WangChu2014,Uschmajew2015,EK2015} specifically deal with the convergence of the higher-order power method.

Finally, we mention the work~\cite{TW2016}, in which global convergence of a related method (called scaled alternating steepest descent) for matrix completion is investigated. 

\bibliographystyle{plain}
\bibliography{subspace.bib}

\end{document}